\documentclass[a4paper, 10pt]{amsart}

\usepackage{filecontents}
\usepackage{float}

\usepackage[english]{babel}
\usepackage[utf8]{inputenc}
\usepackage[T1]{fontenc}

\usepackage{amsmath}
\usepackage{amssymb}

\usepackage{xcolor}

\usepackage[centering,bottom=43mm]{geometry} 

\usepackage{enumerate}
\usepackage{scalerel} 
\usepackage{subfigure}

\usepackage{graphicx}

\usepackage{amsthm}
\newtheorem{theorem}{Theorem}
\numberwithin{theorem}{section}
\newtheorem{lemma}[theorem]{Lemma}

\newtheorem{prop}[theorem]{Proposition}
\newtheorem{cor}[theorem]{Corollary}

\theoremstyle{definition}
\newtheorem{definition}[theorem]{Definition}
\newtheorem{example}[theorem]{Example}
\newtheorem{remark}[theorem]{Remark}
\newtheorem{asn}[theorem]{Assumption}

\newcommand{\N}{\mathbb{N}}

\newcommand{\R}{\mathbb{R}}

\newcommand{\A}{\mathcal{A}}
\newcommand{\M}{\mathcal{M}}
\newcommand{\T}{\mathcal{T}}

\newcommand{\F}{\mathcal{F}}

\newcommand{\B}{\mathcal{B}}

\newcommand{\weaklyto}{\rightharpoonup}

\newcommand{\diam}{\textup{diam}}
\newcommand{\supp}{\textup{supp}}

\newcommand{\prob}{\mathcal{P}}
\newcommand{\pr}{\textup{pr}}
\newcommand{\id}{\textup{id}}
\newcommand{\graph}{\textup{graph}}

\newcommand{\cpl}{\textup{Cpl}}
\newcommand{\cplm}{\textup{Cpl}_{\text{M}}}
\newcommand{\cplbm}{\textup{Cpl}_\text{biM}}
\newcommand{\cplc}{\textup{Cpl}_c}
\newcommand{\cplbc}{\textup{Cpl}_{bc}}

\newcommand{\W}{\mathcal{W}}
\newcommand{\AW}{\mathcal{AW}}

\renewcommand{\epsilon}{\varepsilon}
\renewcommand{\subset}{\subseteq}

\usepackage{scalerel,stackengine}
\stackMath
\newcommand\reallywidehat[1]{%
	\savestack{\tmpbox}{\stretchto{%
			\scaleto{%
				\scalerel*[\widthof{\ensuremath{#1}}]{\kern-.6pt\bigwedge\kern-.6pt}%
				{\rule[-\textheight/2]{1ex}{\textheight}}
			}{\textheight}%
		}{0.5ex}}%
	\stackon[1pt]{#1}{\tmpbox}%
}

\let\oldmarginpar\marginpar
\renewcommand\marginpar[1]{\-\oldmarginpar[\raggedleft\footnotesize #1]{\raggedright\footnotesize\color{red} #1}}
\begin{document}

	\title{Denseness of biadapted Monge mappings}
	\author{Mathias Beiglböck, Gudmund Pammer, Stefan Schrott}

	\thanks{M.\ Beiglböck and S.\ Schrott thank the Austrian Science Foundation FWF for support through projects Y782, P35197 and P34743. All authors are very grateful to Vlad Tuchilus for pointing out an error in Section 3.3.}

	\subjclass[2020]{49Q22, 28A33, 46E27, 60B10}
	\keywords{causal optimal transport, adapted Wasserstein distance, Monge-maps, Kantorovich transport plans}

	\begin{abstract}
		Adapted or causal transport theory aims to extend classical optimal transport from probability measures to  stochastic processes. On a technical level, the novelty is to restrict to couplings which are \emph{bicausal}, i.e.\ satisfy a property which reflects the temporal evolution of information in stochastic processes. We show that in the case of absolutely continuous marginals, the set of bicausal couplings is obtained precisely as the closure of the set of (bi-) adapted processes. That is, we obtain an analogue of the classical result on denseness of Monge couplings in the set of Kantorovich transport plans: bicausal transport plans represent the relaxation of adapted mappings in the same manner as Kantorovich transport plans are the appropriate relaxation of Monge-maps.
	\end{abstract}

	\maketitle

\section{Introduction}\label{sec:intro}

While the main focus lies on an adapted variant of the transport problem, we briefly recount some fundamental elements of the classical theory which will be reinterpreted in a stochastic process  context below.

\subsection{Classical Monge and Kantorovich formulation of the transport problem}

Given probabilities $\mu, \nu$ on $\R^N$ and a lower semicontinuous cost function $c: \R^N\times \R^N \to [0,\infty)$, the classical Monge problem consists in 
\begin{align}\label{Monge} \tag{MP}\inf\left\{\int c(x, T(x)) \, d\mu(x) \;\middle|\; T: \R^N \to \R^N, T_\ast\mu = \nu\right\},\end{align} 
where $T_\ast\mu:= \mu\circ T^{-1} $ denotes the push forward of $\mu$ under a transport map $T$ which is tacitly assumed to be Borel-measurable. While the formulation of the Monge problem allows for a most intuitive interpretation and Monge maps play an important role in many applications, it is paramount for the development of basic transport theory to consider the Kantorovich relaxation of \eqref{Monge}: 
\begin{align}\label{Kantoro} \tag{KP}\inf\left\{\int c(x, y) \, d\pi(x,y) \;\middle|\; \pi \in \cpl(\mu, \nu)\right\},\end{align} 
where $\cpl(\mu, \nu):=\{\pi \in \prob(\R^N\times \R^N)\:|\: \text{proj}_1(\pi) =\mu,  \text{proj}_2(\pi)=\nu\} $ denotes the set of all couplings of $\mu$ and $\nu$. The problem \eqref{Kantoro} is a technically more tractable convex optimization problem, admits a minimizer and allows for a powerful duality theory, we refer to the monographs \cite{Vi03, Vi09, Sa15, FiGl21}. 

Every Monge map $T:\R^N\to \R^N, T_\ast \mu=\nu$ gives rise to a Monge coupling $\pi_T\in \cpl(\mu, \nu)$ concentrated on the graph of the map $T$, i.e.\ $\pi_T=(\text{id}, T)_\ast\mu \in \cpl (\mu, \nu)$. In particular  the value of \eqref{Kantoro} is smaller than or equal to the one obtained in \eqref{Monge}.

More remarkably, the two values coincide under mild regularity assumptions. In fact, for continuous marginals, every Kantorovich transport plan can be approximated by Monge transports: 
\begin{theorem}[cf.\ Gangbo \cite{Ga99} and Ambrosio \cite{Am03}]\label{ClassicDense}
	Assume that $\mu, \nu\in \mathcal P (\R^N) $ are continuous\footnote{We call a measure  continuous if it does not charge singletons.}. Then the set of Monge couplings which are concentrated on the graphs of bijective mappings is dense in  $\cpl (\mu, \nu)$ w.r.t.\ the weak topology.
\end{theorem}
Assuming  continuous marginals, the values of  \eqref{Kantoro} and \eqref{Monge} thus coincide for continuous bounded cost functions. Remarkably, Pratelli \cite{Pr07b} has strengthend this to the case of continuous functions that are allowed to assume the value $+\infty$.

To highlight a particular consequence,  the $p$-Wasserstein distance $\W_p$ (corresponding to the  $c(x,y)= |x-y|^p, p\geq 1$) can be calculated using either the Monge or the Kantorovich formulation. That is, for continuous measures $\mu, \nu$ with finite $p$-th moments we have  \begin{align}\label{Wdistance}
	\W_p^p (\mu, \nu)= \inf_{\pi \in \cpl(\mu, \nu)}\int |x-y|^p \, d\pi(x,y) = \inf_{T :  T \text{ bijective}, T_\ast\mu = \nu}\int |x- T(x)|^p \, d\mu(x).\end{align}
The main goal of this article is to provide counterparts of these results concerning transport between laws of stochastic processes. 

\subsection{Monge- and Kantorovich transport between laws of stochastic processes} 

We are interested in probabilities $\mu, \nu\in \mathcal P(\R^N)$ which represent the laws of real-valued stochastic processes $(X_t)_{t=1}^N$ and $(Y_t)_{t=1}^N$.  In this context, adapted variants of the problems \eqref{Monge} and \eqref{Kantoro} have been considered by different groups of authors, see Section~1.3 for a brief overview. 
A main reason to depart from the classical formulation of optimal transport is that the $p$-Wasserstein distance 
does not yield an adequate topology for typical concepts considered in the theory of stochastic processes. In particular,  basic operations such as the Doob decomposition or the Snell envelope and stochastic control problems or problems of pricing and hedging in mathematical finance are not continuous w.r.t.\ $\W_p$. 

An explanation for this phenomenon is that the classical formulations of optimal transport are agnostic of the temporal evolution of stochastic processes. From a process perspective, transport maps or plans are allowed to  look into the future. Specifically,  transport mappings $T=(T_1, \ldots, T_N): \R^N\to \R^N$ in \eqref{Monge} are free to use all information of the `path' $(x_1, \ldots, x_n)$ to calculate $T_1, T_2, $ etc. In other words, if we consider $T=(T_t)_{t=1}^N$ as a stochastic process on  the space $(\R^N, \mu)$ and denote by $\F_t, t < N$, the $\sigma$-algebra generated by the first $t$ coordinates, then there is no requirement on $T$ to be adapted w.r.t.\ the (canonical) filtration $(\F_t)_{t=1}^N$ in \eqref{Monge}.

This motivates the following definition: 
\begin{definition}
	A map $T=(T_t)_{t=1}^N:\R^N\to \R^N$ is called \emph{adapted}  if $T_t$ depends only on the first $t$ coordinates for $t <  N$. We call $T$ \emph{biadapted} if $T$ is bijective and $T$ as well as $T^{-1}$ are adapted. (Note that if an adapted map $T$ is a bijection, $T^{-1}$ is not necessarily adapted.)
	
	The \emph{biadapted} or \emph{bicausal}  Monge problem consists in 
	\begin{align}\label{biMonge} \tag{MP${}_{\text{ad}}$}\inf\left\{\int c(x, T(x)) \, d\mu(x)\:\middle|\: T: \R^N \to \R^N, T_\ast\mu = \nu, T  \text{ is biadapted}\right\}. 
	\end{align} 
\end{definition}
While \eqref{biMonge} admits a particularly natural interpretation, adapted/causal transport problems have  been introduced directly in a Kantorovich formulation (cf.\ \cite{Ru85, PfPi15, BiTa19, NiSu20} and the comments in  Section~1.3 below):
\begin{definition}
	Let $\mu, \nu \in \prob(\R^N)$ and  $\pi \in \cpl(\mu,\nu)$. Denoting by $(\pi^x)_{x\in \R^N} $ the disintegration of $\pi$ w.r.t.\ $\mu$, the coupling $\pi$ is called \emph{causal} if for any $t < N$ and $B \in \F_t$ the mapping $X \ni x \mapsto \pi^x(B)$ is $\F_t$-measurable. We denote the set of causal couplings between $\mu$ and $\nu$ by $\cplc(\mu,\nu)$.  
	Setting $e: \R^N \times \R^N \to \R^N \times \R^N, (x,y)\mapsto (y,x)$,  
	a causal coupling $\pi \in \cplc(\mu,\nu)$ is called \emph{bicausal} if also $e_\ast\pi \in \cplc(\nu,\mu)$. The set of bicausal couplings between $\mu$ and $\nu$ is denoted by $\cplbc(\mu,\nu)$.
	
	The \emph{biadapted} or \emph{bicausal} Kantorovich problem is given by 
	\begin{align}\label{biKantoro} \tag{KP$_{\text{ad}}$}\inf\left\{\int c(x, y) \, d\pi(x,y) \:\middle|\: \pi \in \cplbc(\mu, \nu)\right\},\end{align} 
	
	A \emph{bicausal Monge coupling} is a bicausal coupling that is supported on the graph of a function.
\end{definition}
Alternatively it would be natural to define a bicausal Monge coupling as a coupling  supported by the graph of a biadapted map. Indeed, this definition leads to the same concept, see Lemma \ref{lem:causal_monge_cpl} in the appendix.


In our main result we reconcile the  Monge and Kantorovich viewpoint of the bicausal transport problem.

\begin{theorem}\label{MainTheorem}
	Assume that $\mu, \nu\in \mathcal P (\R^N)$ are absolutely continuous w.r.t.\ Lebesgue measure. Then the set of biadapted Monge couplings between $\mu$ and $\nu$ is weakly dense in $\cplbc(\mu,\nu)$. In particular, the Monge and Kantorovich formulation agree for continuous bounded cost functions.
\end{theorem}

\begin{remark}\label{rem:intro}
	\begin{enumerate}[(i)]
		\item 
		Let $p \in[1,\infty)$. If $\mu$ and $\nu$ have finite $p$-th moments, we have also $\W_p$-denseness in Theorem~\ref{MainTheorem} (see Theorem~\ref{thm:main}). Therefore,  the adapted analogue $\AW_p$ of the Wasserstein distance between absolutely continuous measures can be defined using either \eqref{biMonge} or \eqref{biKantoro}
		\begin{align*}\label{AWdistance}
			\AW_p^p (\mu, \nu)= \inf_{\pi \in \cplbc(\mu, \nu)}\int |x-y|^p \, d\pi(x,y) = \inf_{T:  T \text{ biadapted}, T_\ast\mu = \nu}\int c(x, T(x)) \, d\mu(x).\
		\end{align*}
		We also note that $\AW_p, p\geq 1$ provides an adequately strong topology to rectify the above mentioned shortcomings of its classical counterpart, e.g.\    Doob decomposition  and optimal stopping are (Lipschitz-) continuous w.r.t.\  $\AW_p$ (see \cite{BaBePa21, AcBaZa20}), operations such pricing, hedging and utility maximization in mathematical finance are continuous, \cite{Do14, GlPfPi17, BaBaBeEd19a, BaDoDo20} etc. On the other hand, it appears that $\AW_p$ is not \emph{overly strong}. E.g.\ the topology generated by $\AW_p$ is essentially the weakest topology which guarantees continuity of optimal stopping problems (\cite[Section~1.5]{BaBaBeEd19b}).

		\item Compared to the classical case, our assumptions in Theorem~\ref{MainTheorem} are stronger in that we require \emph{absolute} continuity of the underlying marginal measures. In fact a (significant) strengthening of continuity is necessary, compare Example \ref{AbsolutelyImportant}.
		
		We also note that while we restrict to optimal transport on $\R^N$ in this introductory section
		to keep notation light, the results on classical transport theory discussed above are equally valid for abstract Polish spaces. Likewise Theorem~\ref{MainTheorem} holds true in the Polish setting (see Theorem \ref{thm:main}), 
		subject to replacing the absolute continuity assumption by Assumption \ref{asn:cont_disint}.

		\item\label{it:rem_intro} An auxiliary result used in the proof of our main Theorem~\ref{MainTheorem} that might be interesting on its own right is that any coupling can be obtained as a projection of a Monge coupling that is supported on the graph a bijection between extended spaces (Theorem~\ref{thm:cplrepr}). The same is true in the bicausal case (Theorem~\ref{thm:cpl_repr_timestep}).

		
	\end{enumerate}
\end{remark}

%
%
%
%

\subsection{Related literature on `adapted' variants  of optimal transport}

Most directly related to the present article is the work \cite{BeLa20} which shows that under appropriate assumptions couplings which are concentrated on the graphs of adapted functions are dense in set of all causal couplings. Remarkably, the methods used in \cite{BeLa20} differ substantially from the ones used in the present article. Specifically, the approach of \cite{BeLa20} aims to reduce the problem to the case where the first marginal admits an approximate product structure. This viewpoint seems to be not useful in the present more rigid case of bicausal transport.   

The use of `causality' constraints between transport plans between  processes  originates in the famous work of Yamada--Watanabe  \cite{YW} and is used under the name `compatibility' by Kurtz \cite{Ku07}. 
A systematic treatment  of  causality (under that name) as an interesting property of abstract transport plans between stochastic processes goes back to   Lassalle \cite{La18}.

Several groups of researchers  from different areas have independently arrived at (roughly) equivalent  adapted variants of the Wasserstein distance, this includes the works of 
Vershik \cite{Ve70, Ve94}, 
R\"{u}schendorf \cite{Ru85},  Gigli \cite[Chapter 4]{Gi04},  Pflug and Pichler \cite{PfPi12}, Bion-Nadal and Talay \cite{BiTa19}, and Nielsen and Sun \cite{NiSu20}. Pflug and Pichler use the name `nested distance' and demonstrate its significant potential in  multistage stochastic programming, see \cite{PfPi14, PfPi15, GlPfPi17, KiPfPi20, PiSh21}  among others. 

Apart from `adapted' Wasserstein distances, 
refinements of the weak topology that take the temporal  flow of information inherent to stochastic processes into account were considered by Aldous (\cite{Al81}, stochastic analysis),   Hoover and Keisler   (\cite{HoKe84, Ho91}, mathematical logic), Hellwig  (\cite{He96}, economics) and Bonnier, Liu, and Oberhauser (\cite{BoLiOb23}, rough path theory).
The use of such extended  weak topologies in mathematical finance starts with Dolinsky \cite{Do14},  further contributions are \cite{GlPfPi17, AcBaZa20, BaDoGu20, BaDoDo20, AcBePa20, BaBaBeEd19a}.

\subsection*{Organization of the paper}
Chapter \ref{sec:static_case} covers the static case: In Section~\ref{subsec:Representation1} we prove the representation result mentioned in Remark~\ref{rem:intro}(\ref{it:rem_intro}), which allows us to  recover the classical denseness result Theorem~\ref{ClassicDense}
from it in Section~\ref{subsec:Denseness1}. This new proof of the well-known denseness result will be extended to a proof of the time-dependent case in Chapter~\ref{sec:time_dependent}. We extend the representation result from the static to the time dependent case in Section~\ref{subsec:Representation2} and use it to prove the denseness result in Section~\ref{subsec:Denseness2}.


\subsection*{Notation}

In the rest of paper we will work (primarily) on abstract Polish spaces rather than on the real line. On the one hand the arguments would be (essentially) identical on the real line and on the other hand the abstract setting is often more convenient in view of notation. We take the liberty to use terms introduced above also in this more general framework and trust that there is no danger of confusion (e.g.\ $\cpl(\mu, \nu)$ will still denote the set of couplings, $\cpl_{bc}(\mu, \nu)$ denotes the set of bicausal couplings between probabilities $\mu, \nu$ on $N$-fold products $X^N, Y^N$ of Polish spaces, we write $e$ for the mapping $X\times Y\to Y\times X, (x,y) \mapsto (y,x)$, etc.

Throughout Polish spaces and standard Borel spaces 
will be denoted by capital letters, such as $X$ or  $Y$. (Recall that a  standard Borel space is a measurable space whose $\sigma$-algebra is induced by a Polish topology.) Collections of subsets of them (e.g.\ topologies or $\sigma$-algebras) will be denoted by calligraphic letters such as $\B$ or $\T$. We will always equip the spaces $X$ and $Y$ with the Borel $\sigma$-algebra generated by their Polish topology. The measurability of mappings is always to be understood  w.r.t.\ the Borel $\sigma$-algebra. A Borel isomorphism is a Borel measurable bijection between standard Borel spaces (its inverse is Borel as well, cf.\ Section~\ref{APrelimSec}). In order to exclude trivial cases, we will assume that Polish spaces and standard Borel spaces are always uncountable (and hence have cardinality of the continuum, see Theorem~\cite[Theorem~13.6]{Ke95}).

Probability measures are denoted with small Greek letters such as $\mu,\nu$ and $\pi$. The Lebesgue measure on $[0,1]$ is denoted by $\lambda$. Given a Borel mapping $f: X \to Y$ and measure $\mu \in \prob(X)$, we denote the push forward of $\mu$ under $f$ as $f_\ast\mu$, i.e.\ $f_\ast\mu(A):= \mu(f^{-1}(A))$ for all $A\subset Y$ Borel. For a further mapping $g : Y \to Z$ we define $g_\ast f_\ast\mu := g_\ast(f_\ast \mu) = (g \circ f)_\ast\mu$ to avoid unnecessary brackets.

We denote the space of probability measures on $X$ by $\prob(X)$ and equip it with the weak convergence, i.e.\ $\mu_n \weaklyto \mu$ if $\int f d \mu_n \to \int f d\mu$ for all $f: X \to \R$ continuous and bounded. 
For a Polish space $X$ with compatible metric $d$ and $p \in [1,\infty)$, let $\prob_p^d(X)$ be the set of $\mu \in \prob(X)$ s.t.\ $\int d^{p}(x_0,x)\ d\mu(x)< \infty$ for some (and therefore any) $x_0 \in X$. On this space we have the $p$-Wasserstein metric
\[ 
\W_p^d(\mu,\nu)^p  := \inf \left\{   \int d(x,y)^p\, d\pi(x,y)  \big| \pi \in \cpl(\mu,\nu)  \right\}.
\]

If it is clear from the context which metric on $X$ we consider, we will just write $\prob_p(X)$ instead of $\prob_p^d(X)$ and $\W_p$ instead of $\W_p^d$.

If $(X,d_X)$ and $(Y,d_Y)$ are metric spaces the product space $X \times Y$ is always equipped with the product metric $d_{X \times Y}^p := d_X^p + d_Y^p$. 

A kernel from $Z$ to $X$ is a Borel function $\pi : Z \to \prob(X)$. We denote the probability measure $\pi(z)$ as $\pi^z$. We introduce a similar notation for functions: Given a function $F : Z \times X \to Y$ (which can also be seen as a function $F: Z \to Y^X$) and $z \in Z$ we define the function $F^z : X \to Y$ as $F^z(x)=F(z,x)$.


\section{The static case}\label{sec:static_case}
The main result of this chapter is that couplings between $\mu$ and $\nu$ that are supported by the graph of a bijection are dense in the set of couplings between $\mu$ and $\nu$, if $\mu$ and $\nu$ are continuous. This result is well known, but we develop a method to prove it, which can be extended to the time-dependent case in order to prove new results in Section~\ref{sec:time_dependent}.

\subsection[Representation of couplings as bijective Monge couplings]{Couplings as projection of couplings supported on the graph of bijections}\label{subsec:Representation1}
In this section we will show that any coupling between $\mu$ and $\nu$ can be obtained as projection of a coupling between $\mu \otimes \lambda$ and $\nu \otimes \lambda$ that is supported on the graph of a bijection (Theorem~\ref{thm:cplrepr}). Before we start, we introduce some notation:
For $\mu \in \prob(X)$ and $\nu \in \prob(Y)$ we write    \[\cplm(\mu,\nu):= \{ (\id,T)_\ast\mu \;|\; T : X \to Y,\: T_\ast\mu=\nu  \} \subset \cpl(\mu,\nu)\]
for the set of Monge couplings between $\mu$ and $\nu$ and 
\[\cplbm(\mu,\nu):= \{ (\id,T)_\ast\mu \;|\; T : X \to Y \text{ bijective, } T_\ast\mu=\nu  \} \subset \cplm(\mu,\nu)\]
for the set of bijective Monge couplings between $\mu$ and $\nu$.

A relation $R \subset X \times Y$ is the graph of a bijection from $X$ to $Y$ if and only if $R$ is the graph of a mapping from $X$ to $Y$ and the inverse relation $R^{-1} \subset Y \times X$ is the graph of a mapping from $Y$ to $X$. It is straightforward to see that the same is true for couplings, i.e.\ we have
$\pi \in \cplbm(\mu,\nu)$ if and only if $\pi \in \cplm(\mu,\nu)$ and $e_\ast\pi \in \cplm(\nu,\mu)$.

The aim  of this section is to show the following:

\begin{theorem}\label{thm:cplrepr}
	Let $X$ and $Y$ be standard Borel spaces, $\mu \in \prob(X)$, $\nu \in \prob(Y)$ and let $\pi \in \cpl(\mu,\nu)$. Then there exists a measurable bijection $T: X \times [0,1] \to Y \times [0,1]$ satisfying the following properties:
	\begin{enumerate}[(i)]
		\item $T_\ast (\mu \otimes \lambda) = \nu \otimes \lambda$
		\item Writing $pr_{XY}: X \times [0,1] \times Y \times [0,1] \to X \times Y$ for the projection onto $X\times Y$ and $\widehat\pi := (id,T)_\ast(\mu \otimes \lambda)$, we have  ${pr_{XY}}_\ast\widehat\pi=\pi$.
	\end{enumerate}
\end{theorem}

Loosely speaking, the theorem says that it suffices to consider bijective Monge couplings with the trade-off of replacing the spaces $X$ and $Y$ by the bigger spaces $X \times [0,1]$ and $Y \times [0,1]$.

\begin{remark} \label{rem:easyy_repr}
	Before we start the proof, we note that a version of this result is well  known for (not necessarily bijective) Monge couplings (c.f. {\cite[Lemma~3.22]{Kall02}}): For any $\pi \in \cpl(\mu,\nu)$ there is a $\widehat{\pi} \in \cplm(\mu \otimes \lambda, \nu)$ s.t.\ we can recover $\pi$ from $\widehat \pi$ when projecting from $X \times [0,1] \times Y$ onto $X \times Y$. 
	
	To see this, assume $X=Y=[0,1]$ (c.f.\ Section~\ref{APrelimSec}) and set  $T(x,u):= F_{\pi^x}^{-1}(u)$. Clearly, $T(x,\cdot)_\ast\lambda=\pi^x$. If $f: X \times Y \to \R$ is a Borel function, we see that
	\begin{align*}
		\int \!\! f(x,y)d (\id,T)_\ast(\mu \otimes \lambda) = \!\!\int\!\! f(x,F_{\pi^x}^{-1}(u)) d\lambda(u)d\mu(x) = \!\!\int\!\! f(x,y) d\pi^x(y)d\mu(x)=\!\!\int\!\!  f(x,y) d\pi(x,y),
	\end{align*}
	so ${\pr_{XY}}_\ast T_\ast(\mu\otimes\lambda)=\pi$ and hence $T_\ast(\mu \otimes \lambda) \in \cpl(\mu\otimes\lambda,\nu)$. 
\end{remark}

If $\pi =(\id,T)_\ast\mu \in \cpl(\mu,\nu)$ is a Monge coupling, the conditional probabilities w.r.t.\ the first coordinate are Dirac measures, i.e.\ $\pi^x=\delta_{T(x)}$. However, the conditional probabilities w.r.t.\ the second coordinate do not need to be Dirac, unless the Monge mapping is injective. So, in a certain sense a Monge coupling can still contain randomness (given some $y \in Y$ one can in general not determine ``from which $x$ the mass in $y$ came''). Hence, it is reasonable 
to enlarge both $X$ and $Y$ to $X \times [0,1]$ and $Y \times [0,1]$ in order to obtain a representation of $\pi \in \cpl(\mu,\nu)$ as a coupling that is concentrated on the graph of a bijection.

The idea of the proof of Theorem~\ref{thm:cplrepr} is to consider the following mapping:
\[
T : X \times [0,1] \to Y \times [0,1] : (x,u) \mapsto (y,v), \text{ where } y=F_{\pi^x}^{-1}(u), v=F_{\pi^y}(x).
\]
As in Remark \ref{rem:easyy_repr}, we calculate $y$ given some $x$ and the 
extra parameter $u$ via $y=F_{\pi^x}^{-1}(u)$. For the definition of $v$ we observe that this parameter belongs to the $Y$-component and that we already have a prescribed value for $y$. So, we have to ask: ``given $y$ and knowing that the result of our calculation is $x$, what is the suitable value for $v$?'' Therefore, $v$ should satisfy $x=F_{\pi^{y}}^{-1}(v)$, so formally $v=F_{\pi^y}(x)$. Moreover, it is easy to check that the mapping\footnote{One has to be slightly careful when reading the definition of $S$: Basically, $S$ is the same mapping as $T$ but for the coupling $e_\ast\pi \in \cpl(\nu,\mu)$, where $e(x,y):=(y,x)$. Hence, the $F_{\pi^y}^{-1}$ in definition of $S$ are the quantile functions of $\pi$ conditioned on some $y \in Y$, i.e.\ we do not have just changed the names of the variables $x$ and $y$ when defining an inverse function, in fact we disintegrate w.r.t.\ to another coordinate than in the definition of $T$. 
}
\[
S : Y \times [0,1] \to X \times [0,1] : (y,v) \mapsto (x,u), \text{ where } x=F_{\pi^y}^{-1}(v), u=F_{\pi^x}(y)
\] 
is the inverse of $T$ provided that $F_{\pi^x}$ and $F_{\pi^y}$ are bijective for  all $x \in X$ and $y \in Y$. However, this is only true for couplings $\pi$, whose conditional probabilities $\pi^x$ and $\pi^y$ are non-atomic. 

We will prove Theorem~\ref{thm:cplrepr} with a slightly more elaborate version of the construction presented above to overcome this issue.  Moreover, we will prove a parameterized version of this theorem in order to avoid measurability issues later on.

\begin{theorem}\label{thm:cpl_repr}
	Let $X, Y, Z$ be standard Borel spaces and $\pi$ a kernel from $Z$ to $X\times Y$. Denote $\mu$ the kernel from $Z$ to $X$ defined by $\mu^z := {\pr_X}_\ast \pi^z$ and $\nu$ the kernel from $Z$ to $Y$ defined by $\nu^z := {\pr_Y}_\ast \pi^z$, i.e.\ $\pi^z \in \cpl(\mu^z,\nu^z)$ for all $z \in Z$.
	
	Then there exists a Borel mapping $T : Z \times X \times [0,1] \to Y \times [0,1]$ s.t.\ for all $z \in Z$ the mappings $T^z: X \times [0,1] \to Y \times [0,1] : (x,u) \mapsto T(z,x,u)$ are Borel isomorphisms satisfying
	\begin{enumerate}[(i)]
		\item ${T^z}_\ast (\mu^z \otimes \lambda) = \nu^z \otimes \lambda$
		\item ${\pr_{XY}}_\ast (id,T^z)_\ast (\mu^z \otimes \lambda)= \pi^z$.
	\end{enumerate}
\end{theorem}

\begin{proof}
	By Corollary \ref{cor:borel_iso_parametr} there exists a measurable mapping
	\[
	G: (Z \times X) \times Y \times [0,1] \to  [0,1]^2
	\] 	
	s.t.\ for all $(z,x) \in Z \times X$ the mapping $G^{z,x}:=G(z,x,\cdot) : Y \times [0,1] \to [0,1]^2$ is a Borel isomorphism satisfying  $G^{z,x}_\ast(\pi^{z,x} \otimes \lambda ) = \lambda^2$.
	
	Again by Corollary \ref{cor:borel_iso_parametr},  there exists a measurable mapping
	\[
	H : (Z \times Y) \times X \times [0,1] \to [0,1]^2
	\]
	s.t.\ for all $(z,y) \in Z \times Y$ the mapping $H^{z,y}:=H(z,y,\cdot) : X \times [0,1] \to [0,1]^2$ is a Borel isomorphism satisfying  $H^{z,y}_\ast(\pi^{z,y} \otimes \lambda ) = \lambda^2$.

	Consider the mapping $S : Z \times X \times [0,1]^3 \to Y \times [0,1]^3$ defined by $S(z,x_1,x_2,u_1,u_2)=(y_1,y_2,v_1,v_2)$, where
	\[
	(y_1,y_2)=(G^{z,x_1})^{-1}(u_1,u_2) \qquad (v_1,v_2)=H^{z,y_1}(x_1,x_2).
	\]
	Clearly, $S$ is Borel. For $z \in Z$ we denote $S^z := S(z,\cdot) : X \times [0,1]^3 \to Y \times [0,1]^3$.
	
	Our aim is to show that for all $z \in Z$ the mapping $S^z$ is a Borel isomorphism satisfying 
	\begin{enumerate}[(i)]
		\item ${S^z}_\ast (\mu^z \otimes \lambda^3) = \nu^z \otimes \lambda^3$,
		\item ${\pr_{XY}}_\ast (\id,S^z)_\ast (\mu^z \otimes \lambda^3)= \pi^z$.
	\end{enumerate}
	In order to prove the injectivity of $S^z$, let $(x_1,x_2,u_1,u_2) \neq (\bar x_1, \bar x_2, \bar u_1, \bar u_2)$ be given. We have to show that $(y_1,y_2,v_1,v_2) := S^z(x_1,x_2,u_1,u_2)$ and $(\bar y_1, \bar y_2, \bar v_1, \bar v_2) := S^z(\bar x_1, \bar x_2, \bar u_1, \bar u_2)$ are different. In the case $(y_1,y_2) \neq (\bar y_1, \bar y_2)$ there is nothing to prove, so we may assume $(y_1,y_2) = (\bar y_1, \bar y_2)$. We distinguish two cases:
	
	Case 1: $x_1 = \bar x_1$. This implies $(G^{z,x_1})^{-1} = (G^{z,\bar x_1})^{-1}$ and by the injectivity of this mapping we get $(u_1,u_2)=(\bar u_1, \bar u_2)$. Since $(x_1,x_2,u_1,u_2) \neq (\bar x_1, \bar x_2, \bar u_1, \bar u_2)$ this implies $x_2 \neq \bar x_2$ and by the injectivity of $H^{z,y_1}$ this implies $(v_1,v_2) \neq (\bar v_1, \bar v_2)$.
	
	Case 2: $x_1 \neq \bar x_1$. Then by the injectivity of $H^{z,y_1}$ we have $(v_1,v_2) \neq (\bar v_1, \bar v_2)$ as well.
	
	For proving the surjecitvity, let $(y_1,y_2,v_1,v_2)$ be given. By the surjectivity of $H^{z,y_1}$ there are $(x_1,x_2)$ such that $H^{z,y_1}(x_1,x_2)=(v_1,v_2)$. Now, by the surjectivity of $(G^{z,x_1})^{-1}$, there exists $(u_1,u_2)$ such that $(G^{z,x_1})^{-1}(u_1,u_2)=(y_1,y_2)$.
	
	We have shown that $S^z$ is a Borel measurable bijection and thus a Borel isomorphism (c.f.\ Section~\ref{APrelimSec}).
	
	\textit{Property (i).} Let $f :Y \times [0,1]^3$ be a Borel function. Then it holds
	\begin{align*}
		\int f(y_1,y_2,&v_1,v_2) dS^z_\ast(\mu^z \otimes \lambda^3)(y_1,y_2,v_1,v_2) = \\
		=& \int f( (G^{z,x_1})^{-1}(u_1,u_2), H(z,\pr_Y((G^{z,x_1})^{-1}(u_1,u_2)),x_1,x_2)) d(\mu^z \otimes \lambda^3)(x_1,x_2,u_1,u_2)\\
		=& \int f(y_1,y_2,H^{z,y_1}(x_1,x_2) ) \underbrace{d (G^{z,x_1})^{-1}_\ast \lambda^2(y_1,y_2)}_{= d(\pi^{z,x_1}\otimes \lambda)(y_1,y_2)}d(\mu^z\otimes\lambda)(x_1,x_2)\\
		=& \int f(y_1,y_2,H^{z,y_1}(x_1,x_2) ) \underbrace{d\pi^{z,x_1}(y_1)d\mu^z(x_1)}_{\substack{=d\pi^z(x_1,y_1) \\=d\pi^{z,y_1}(x_1)d\nu^z(y_1) }}d\lambda^2(x_2,y_2)\\
		=& \int f(y_1,y_2,H^{z,y_1}(x_1,x_2)) d\pi^{z,y_1}(x_1)d\lambda(x_2)d\nu^z(y_1)d\lambda(y_2)\\
		=& \int f(y_1,y_2,v_1,v_2)\underbrace{ dH^{z,y_1}_\ast(\pi^{z,y_1}\otimes\lambda)(v_1,v_2)}_{=d\lambda^2(v_1,v_2)}d\nu^z(y_1)d\lambda(y_2)\\
		=& \int f(y_1,y_2,v_1,v_2) d(\nu^z\otimes\lambda^3)(y_1,y_2,v_1,v_2),
	\end{align*} 
	which yields $S^z_\ast(\mu^z\otimes\lambda^3)=\nu^z\otimes\lambda^3$.
	
	\textit{Property (ii).} Note that $(\pr_X \circ (G^{z,x_1})^{-1})_\ast\lambda^2=\pi^{z,x_1}$ and that 
	\[
	\pr_{XY} \circ (\id_{X \times [0,1]^3},S^z) : X \times [0,1]^3\to X\times Y : (x_1,x_2,u_1,u_2) \mapsto (x_1,\pr_Y((G^{z,x_1})^{-1}(u_1,u_2))).
	\]
	Hence, for any Borel function $f : X \times Y \to \R$ we have
	\begin{align*}
		\int f(x_1,y_1) d{\pr_{XY}}_\ast& (\id_{X \times [0,1]^3},S^z)_\ast (\mu^z\otimes\lambda^3)(x_1,y_1) = \\
		=&\int f(x_1,\pr_Y((G^{z,x_1})^{-1}(u_1,u_2))) d\mu^z(x_1) d\lambda^3(x_2,u_1,u_2)\\
		=&\int f(x_1,y_1) \underbrace{d(G^{z,x_1})^{-1}_\ast\lambda^2(u_1,u_2)
		}_{=d\pi^{z,x_1}(y_1)d\lambda(y_2)}d\mu^z(x_1)\\
		=&\int f(x_1,y_1) d\pi^{z,x_1}(y_1)d\mu^z(x_1)\\
		=& \int f(x_1,y_1) d\pi^z(x_1,y_1),
	\end{align*}
	which shows that ${\pr_{XY}}_\ast(\id_{X  \times [0,1]^3},S^z)_\ast(\mu^z \otimes \lambda^3)=\pi^z$.
	
	By the Borel isomorphism theorem there exists a Borel isomorphism $h : [0,1] \to [0,1]^3$ satisfying $h_\ast\lambda=\lambda^3$. Define\footnote{To clarify the notation, for $f: A \to B$ and $g: A \to C$ we define $(f,g) : A \to B \times C : a \mapsto (f(a),g(a))$. For $f : A \to B$ and $g : C \to D$ we define $f \times g : A \times C \to B \times D : (a,c) \mapsto (f(a),g(c))$. }
	\[
	T := ( \id_Y \times h^{-1} )\circ S \circ (\id_Z \times \id_X \times h) : Z \times X \times [0,1] \to Y \times [0,1] 
	\]
	Clearly, $T$ is measurable as composition. For $z \in Z$ it holds $T^z = ( \id_Y \times h^{-1} )\circ S^z \circ ( \id_X \times h)$, so $T^z$ is a Borel isomorphism as composition of Borel isomorphisms. Moreover, it is easy to check that $T^z_\ast(\mu^z\otimes\lambda) = \nu^z \otimes \lambda$ and that $\widehat{\pi}:=(\id,T)_\ast(\mu \otimes \lambda)$ satisfies ${\pr_{XY}}_\ast\widehat{\pi^z}=\pi^z$.
\end{proof}

\subsection{Denseness of couplings supported by the graph of a bijection}\label{subsec:Denseness1}
The aim of this section is to prove that $\cplbm(\mu,\nu)$ is dense in $\cpl(\mu,\nu)$ if $\mu$ and $\nu$ are continuous measures. For that purpose, we need to approximate a given coupling $\pi \in \cpl(\mu,\nu)$ by a sequence $(\pi_n)_n$ in $\cplbm(\mu,\nu)$. 

To that end, we use the representation of $\pi$ as $\widehat{\pi}=(\id,T)_\ast(\mu \otimes \lambda) \in \cplbm(\mu\otimes \lambda, \nu \otimes \lambda)$ from the previous section. Then we choose sequences of partitions (with mesh converging to zero) of the spaces $X$ and $Y$  and bijections between $X$ and $X \times [0,1]$ (and respectively between $Y$ and $Y \times [0,1]$), which are compatible with these partitions (Proposition~\ref{prop:part_pushfwd}). In the proof of Theorem~\ref{thm:cplbm_dense} we show that the concatenation of these compatible bijections and $T$ is a suitable approximating sequence.

For a partition $\M$ of a metric space we define its mesh as $||\M|| := \sup_{M \in \M} \diam(M)$.
The following straightforward fact will be used below:
Let $X$ be a Polish space and $d$ be a compatible metric. 
Then there exists a sequence $(\M_n)_{n \in \N}$ of partitions of $X$ satisfying $ \lim_{n \to \infty}||\M_n|| = 0$ such that each partition $\M_n$ consists of at most countably many Borel subsets of $X$. 

The following sufficient criterion for weak convergence is convenient for proving  convergence of the approximating sequence that we construct in the proof of Theorem~\ref{thm:cplbm_dense}.
\begin{lemma}\label{lem:part_conv}
	Let $X$ be a Polish space with a compatible metric $d$, and let $(\M_n)_{n \in \N}$ be a sequence of partitions of $X$ satisfying $ \lim_{n \to \infty}||\M_n|| = 0$ such that each partition $\M_n$ consists of at most countably many Borel subsets of $X$. Assume that $\mu_n,\mu \in \prob(X)$  satisfy $\mu_n(M)=\mu(M)$ for all $M \in \M_n$. Then $\mu_n \to \mu$ w.r.t.\ weak convergence. 
	
	Let $p \in [1,\infty)$ and assume that  $\mu_n, \mu$ have finite $p$-th moments. Then we have $\mu_n \to \mu$ in $\W_p$ as well.
\end{lemma}

\begin{proof}
	We first cover the case $\mu_n,\mu \in \prob_p(X)$. Since $\mu_n(M)=\mu(M)$ for all $M \in \M_n$ the probability measure
	\[
	\pi_n := \sum_{M \in \M_n} \frac{1}{\mu(M)} \mu_n|_M \otimes \mu|_M.
	\]
	is a coupling between $\mu_n$ and $\mu$. Using the coupling $\pi_n$ to estimate the Wasserstein distance of $\mu_n$ and $\mu$, we obtain
	\begin{align*}
		\W_p^p(\mu_n,\mu) &\le 	
		\int d^p\: d\pi_n  = \sum_{M \in \M_n} \int_M d^p\:d\pi_n  \le \sum_{M \in \M_n}\diam(M)^p \pi_n(M\times M) \\
		&\le ||\M_n||^p \sum_{M \in \M_n} \mu(M) = ||\M_n||^p \to 0,
	\end{align*}
	which implies $\W_p$- and hence weak convergence. If merely $\mu_n,\mu \in \prob(X)$ we can replace $d$ by $\widehat{d}:= \min(d,1)$. Then $\mu_n,\mu \in \prob_p^{\widehat{d}}(X)$ and the above consideration together imply that $\mu_n \to \mu$ weakly.
\end{proof}

The following proposition is (up to a few technicalities) a consequence of the isomorphism theorem for measures (cf.\ Section~\ref{APrelimSec}), which states that for any two continuous probability measures, there exists a bijection that pushes the first measure to the second.

\begin{prop}\label{prop:part_pushfwd}
	Let $X$ be a Polish space, $\M$ be an at most countable partition of $X$ consisting of Borel sets and $\mu \in \prob(X)$ be continuous. Then there exists a Borel isomorphism $\Phi_\mu^\M : X \to X \times [0,1]$ such that for all $M \in \M$ it holds $(\Phi_\mu^\M) _\ast(\mu|_M)=(\mu|_M) \otimes \lambda$.
\end{prop}
\begin{proof}
	See Proposition~\ref{prop:part_pushfwd2} in the appendix.
\end{proof}

Now we are ready to prove the main theorem of this section:
\begin{theorem}\label{thm:cplbm_dense}
	Let $X, Y$ be standard Borel spaces and  let $\mu \in \prob(X)$, $\nu \in \prob(Y)$ be continuous. Then $\cplbm(\mu,\nu)$ is weakly dense in $\cpl(\mu,\nu)$.
	
	Let $p \in [1,\infty)$. If $\mu$ and $\nu$ have finite $p$-th moments (w.r.t. compatible metrics $d_X, d_Y$), we have $\W_p$-denseness as well.
\end{theorem}
\begin{proof}
	Note that if $\mu$ and $\nu$ both have finite $p$-th moments, then every $\pi \in \cpl(\mu,\nu)$ has finite $p$-th moments w.r.t.\ the product metrc. As $\cpl(\mu,\nu)$ is weakly and $\W_p$-closed, the weak- and $\W_p$-closures of $\cplbm(\mu,\nu)$  are contained in $\cpl(\mu,\nu)$. In order to show that the closure of $\cplbm(\mu,\nu)$ is $\cpl(\mu,\nu)$, we have to show that any $\pi \in \cpl(\mu,\nu)$ can be approximated by a sequence $(\pi_n)_n$ in $\cplbm(\mu,\nu)$.
	
	According to Theorem~\ref{thm:cplrepr} there exists a coupling $\widehat{\pi} = (\id,T)_\ast(\mu\otimes\lambda) \in \cplbm(\mu\otimes \lambda, \nu \otimes \lambda)$ such that
	\begin{enumerate}[(i)]
		\item $T_\ast(\mu \otimes\lambda)=\nu \otimes \lambda$,
		\item ${\pr_{XY}}_\ast\widehat{\pi}=\pi$.
	\end{enumerate}
	
	Let $(\A_n)_{n \in \N}$ and $(\B_n)_{n\in\N}$ be sequences of partitions of $X$ and $Y$ consisting of countably many Borel sets and satisfying $ \lim_{n \to \infty}||\A_n|| = 0$ and $ \lim_{n \to \infty}||\B_n|| = 0$. According to Proposition~\ref{prop:part_pushfwd}, for any $n \in \N$ there exist bijections $\Phi_n : X \to X \times [0,1]$ and $\Psi_n : Y \to Y \times [0,1]$ such that 
	\begin{enumerate}[(i)]
		\setcounter{enumi}{2}
		\item ${\Phi_n}_\ast(\mu|_A)=(\mu|_A)\otimes \lambda$ for all $A \in \A_n$,
		\item ${\Psi_n}_\ast(\nu|_B)=(\nu|_B)\otimes \lambda$ for all $B \in \B_n$.
	\end{enumerate}
	
	For $n \in \N$ define the mapping
	\[
	T_n := \Psi_n^{-1} \circ T\circ\Phi_n : X \to Y.
	\]
	It is easy to see that $T_n$ is bijective and satisfies ${T_n}_\ast\mu=\nu$. We need to check that $\pi_n:=(\id,T_n)_\ast\mu \weaklyto \pi$. Note that $\A_n \otimes \B_n := \{ A \times B \:|\:\ A \in \A_n, B \in \B_n \}$ are partitions of $X \times Y$ consisting of countably many Borel sets satisfying $\lim_{n \to \infty} ||\A_n \otimes \B_n ||=0$. Hence, by Lemma \ref{lem:part_conv} it suffices to show for all $n \in \N$ and for all $ A \in \A_n\: \forall B \in \B_n$
	\[
	\pi_n(A\times B) = \pi(A \times B).
	\]
	This is a consequence of the properties (i)  to (iv) of the mappings $\Phi_n,\Psi_n$ and $T$:
	\begin{align*}
		\pi_n(A\times B) &= \mu(A \cap T_n^{-1}(B)) = \mu|_A((\Phi_n^{-1}\circ T^{-1} \circ \Psi_n)(B))  \stackrel{(iii)}{=} (\mu|_A \otimes \lambda)(T^{-1}(\Psi_n(B)))\\
		&= (\mu \otimes \lambda)((A \times [0,1]) \cap T^{-1}(\Psi_n(B)) ) \stackrel{(i)}{=} (\nu \otimes \lambda)(T(A \times [0,1]) \cap \Psi_n(B) )\\
		&= (\nu \otimes \lambda)(\Psi_n( \Psi_n^{-1}(T(A\times [0,1])) \cap B  )) = {\Psi_n^{-1}}_\ast(\nu\otimes\lambda) ( \Psi_n^{-1}(T(A\times [0,1]))\cap B)\\
		&\hspace*{-0.25em}\stackrel{(iv)}{=}\nu|_B(\Psi_n^{-1}(T(A\times [0,1]))) \stackrel{(iv)}{=}(\nu|_B \otimes \lambda)(T(A\times [0,1]))\\\
		&=(\nu\otimes\lambda)(T(A\times [0,1]) \cap (B \times [0,1]) ) = (\id,T^{-1})_\ast(\nu \otimes \lambda)(B \times [0,1] \times A \times [0,1])\\
		&= \widehat{\pi}(A \times [0,1] \times B \times [0,1]) \stackrel{(ii)}{=} \pi(A \times B).
	\end{align*}
\end{proof}

\begin{remark}
	It is well known that $\cplm(\mu,\nu)$ is dense in $\cpl(\mu,\nu)$ if $\mu$ is continuous. We can prove this result with little effort using the tools that we have developed so far. The proof can be carried out exactly as the proof of Theorem~\ref{thm:cplbm_dense} with one important exception: If $\nu$ has atoms, the existence of the bijections $\Psi_n : Y \to Y \times [0,1]$ that push $\nu$ to $\nu \otimes \lambda$ (and are compatible with the given partition) fails. However, if we just replace $\Psi_n^{-1}$ by $\pr_Y : Y \times [0,1] \to Y$ in the definition of the mappings $T_n$, the only property of $T_n$ that we loose is its injectivity. Hence, we have constructed a sequence of mappings $(T_n)_{n \in \N}$ that push $\mu$ to $\nu$ s.t.\ $(\id,T_n)_\ast\mu$ converges to the given coupling $\pi$.
\end{remark}


\section{The time dependent case}\label{sec:time_dependent}
The aim of this chapter is to extend the results of the previous chapter to the time dependent case. Then main result is that a coupling between the laws of two stochastic processes can be approximated by biadapted Monge mappings if and only if it is bicausal (see Theorem~\ref{thm:main} for the exact statement including the regularty assumptions on the marginals).



Before we start, we fix some notation for this chapter: $N \in \N$ will always be the number of time steps that we consider. $X_1 \dots, X_N$, $Y_1,\dots,Y_N$ are always Polish spaces.

$\prod_{i=1}^N X_i$ will be the path space of the first process, whose law will be denoted by $\mu$ and 
$\prod_{i=1}^N Y_i$ will be the path space of the second process, whose law will be denoted by $\nu$.

For $1 \le s <t \le N$ we introduce the abbreviation  $X_{s:t} := \prod_{i=s}^{t} X_i$. We use the same abbreviation for elements of $X_{s:t}$, i.e.\ $(x_s,x_{s+1},\dots,x_t)=:x_{s:t}$, and for subsets, i.e.\ $A_s \times A_{s+1} \times \dots\times A_t =: A_{s:t}$ for $A_i \subset X_i$. We use $X$ as a shorthand for $X_{1:N}$. For the $Y$-component we use analogous notations.

For $t \le N$ define $\F_t^X$ as the $\sigma$-algebra on $X \times Y$ generated by the projections $ X\times Y \to X_{1:t} : (x,y) \mapsto x_{1:t}$ (and $\F_t^Y$ respectively).

We will often decompose $\mu \in \prob(X)$ as $\mu(dx)=\mu_1(dx_1)\mu^{x_1}(dx_{2:N})$, where $\mu_1 \in \prob(X_1)$ and $x_1\mapsto \mu^{x_1}$ is a kernel from $X_1$ to $X_{2:N}$. Iterating this yields
\[
\mu(dx)=\mu_1(dx_1)\mu^{x_1}(dx_2)\cdots \mu^{x_{1:N-1}}(dx_N),
\]
i.e.\ $\mu_1 \in \prob(X_1)$ and for all $t<N$ there are kernels $x_{1:t} \mapsto\mu^{x_{1:t}}$ from $X_{1:t}$ to $X_{t+1}$.

\subsection{Biadapted mappings and bicausal couplings}\label{sec:Causal}
%
%

The following lemma is an inductive characterization of biadapted mappings that will be helpful below.

\begin{lemma}\label{lem:indCharOfBiadapted}		
	Let $T_1 : X_1 \to Y_1$ be a bijection and let $S: X_1 \times X_{2:N} \to Y_{2:N}$ be a Borel mapping such that for all $x_1 \in X_1$ the mapping $S^{x_1} :  X_{2:N} \to Y_{2:N} : x_{2:N} \mapsto S(x_1,x_{2:N})$ is biadapted.
	
	Then the mapping
	\[
	T : X_{1:N} \to Y_{1:N} : x_{1:N} \mapsto (T_1(x_1), S^{x_1}(x_{2:N}))
	\]
	is biadapted.
\end{lemma}
\begin{proof}
	See Appendix \ref{subsec:proofs}.
\end{proof}
%

\begin{prop}[{\cite[Proposition~5.1]{BaBeLiZa17}}]\label{prop:bicausality_crit}
	Let $\mu \in \prob(X), \nu \in \prob(Y)$ and $\pi \in \prob(X\times Y)$. Then the following are equivalent:
	\begin{enumerate}[(i)]
		\item $\pi \in \cplbc(\mu,\nu)$;
		\item When decomposing 
		\[
		\pi(dx,dy)=\pi_1(dx_1,dy_1)\pi^{x_1,y_1}(dx_2,dy_2)\cdots \pi^{x_{1:N-1},y_{1:N-1}}(dx_N,dy_N)
		\]
		we have
		\begin{enumerate}[(a)]
			\item $\pi_1 \in \cpl({\pr_1}_\ast\mu,{\pr_1}_\ast\nu)$
			\item $
			\pi^{x_{1:t},y_{1:t}} \in \cpl(\mu^{x_{1:t}},\nu^{y_{1:t}})
			$ for all $t<N$ and $\pi$-almost all $(x_{1:t},y_{1:t})$.
			
		\end{enumerate}
	\end{enumerate}
	
\end{prop}
An easy corollary of this criterion is the following inductive characterization of bicausal couplings:
\begin{cor}\label{cor:inductiveCharBicausal}
	Let $\mu \in \prob(X), \nu \in \prob(Y)$ and $\pi \in \cpl(\mu,\nu)$. Then $\pi \in \cplbc(\mu,\nu)$ if and only if, when decomposing $\pi$ as $d\pi(x,y)=d\pi_1(x_1,y_1)d\pi^{x_1,y_1}(x_{2:N},y_{2:N})$, one has $\pi^{x_1,y_1} \in \cplbc(\mu^{x_1},\nu^{y_1} )$ for $\pi_1$-almost all $(x_1,y_1)$.
\end{cor}


\begin{prop}[{\cite[Theorem~3.1]{BeLa20}}]\label{prop:cplbc_compact}
	$\cplbc(\mu,\nu)$ is closed w.r.t.\ weak convergence.
\end{prop}

\subsection{Time-dependent version of the representation of couplings}\label{subsec:Representation2}

Our next aim is to prove a time-dependent version of the representation of couplings as bijective Monge couplings on extended spaces. To that end, we will have to add an additional coordinate for randomization in each time step in the $X$- and $Y$-component. To keep notations short, we introduce the following abbreviations:
\begin{align}\label{eq:notation1}
	\widehat{X}_t := X_t \times [0,1], \qquad \widehat{X}_{s:t} := \prod_{i=s}^t\widehat{X}_i, \qquad \widehat{X}:=\widehat{X}_{1:N}.
\end{align}

We will not always be careful about the ordering of the spaces $X_i$ and $[0,1]$ in the definition of $\widehat{X}_{s:t}$ as product of those spaces. Instead of this, we agree to use consistent letters to name elements of those spaces unambiguously: Elements of $\widehat{X}_i$ are always called $(x_i,u_i)$, where $x_i \in X_i$ and $ u_i \in [0,1]$.

Therefore, $((x_s,u_s),\dots,(x_t,u_t))$ denotes the same element of $\widehat{X}_{s:t}$ as $(x_s,\dots,x_t,u_s,\dots,u_t)$ does, and the latter is often abbreviated by $(x_{s:t},u_{s:t})$. When evaluating functions $f : \widehat{X}_{s:t} \to \R$, we use the same convention, i.e.\ $f(x_{s:t},u_{s:t}):= f(x_s,\dots,x_t,u_s,\dots,u_t) := f((x_s,u_s),\dots,(x_t,u_t))$.

$\pr_X$ will always denote the projection $\widehat{X}_{s:t}  \to X_{s:t} : (x_{s:t},u_{s:t}) \mapsto x_{s:t}$.

For $\mu \in \prob(X)$ we define $\widehat{\mu} \in \prob(\widehat{X})$ via
\[
\int f d \widehat{\mu}:= \int f(x_{1:N},u_{1:N}) d\mu(x_{1:N})d\lambda(u_{1:N}),
\]
i.e. we have $\widehat\mu:= \mu \otimes \lambda^N$ when allowing for the minor abuse of notation to rearrange the factors in the product space $\widehat X$ to  $(\prod_{t=1}^N X_t)\times [0,1]^N$.

We use the same convention for the $Y$-component, where elements of $\widehat{Y_t}$ are called $(y_t,v_t)$ with $y_t \in Y_t$ and $v_t \in [0,1]$, etc.

The terms (bi)adapted and (bi)causal are always meant to be understood as the ordering of the spaces in \eqref{eq:notation1} suggests: In each time step we consider the spaces $\widehat{X}_t=X_t \times [0,1]$ and $\widehat{Y_t}=Y_t \times [0,1]$. Loosely speaking, \textit{in} (not before or after)  each time step we ``add'' one unit interval in the $X$-component and one unit interval in the $Y$-component. Explicitly,  a mapping $T: \widehat{X} \to \widehat{Y}$ is adapted if for all $t \le N$ there exists mappings $T_t : \widehat{X}_{1:t} \to \widehat{Y}_t$ s.t.\ $T(x_{1:N},u_{1:N})=(T_1(x_1,u_1),\dots, T_N(x_{1:N},u_{1:N}))$.

Using this notation, we can formulate the main result  of this subsection:
\begin{theorem}\label{thm:cpl_repr_timestep}
	Let $\mu \in \prob(X)$ and  $\nu \in \prob(Y)$. If $\pi \in \cplbc(\mu,\nu)$, then there exists a biadapted mapping  $T : \widehat{X} \to \widehat{Y}$ satisfying
	\begin{enumerate}[(i)]
		\item $T_\ast \widehat{\mu}=\widehat{\nu}$, or equivalently, $\widehat{\pi}:= (id,T)_\ast\widehat\mu \in \cplbc(\widehat{\mu},\widehat{\nu})$,
		\item ${\pr_{X\times Y}}_\ast\widehat{\pi}=\pi$.
	\end{enumerate}
	On the other hand, if $\pi \in \prob(X \times Y)$ and there exists a  biadapted mapping  $T : \widehat{X} \to \widehat{Y}$ such that (i) and (ii) are satisfied, then $\pi \in \cplbc(\mu,\nu)$.
\end{theorem}
We postpone the (easy) proof of the  reverse implication to the end of this section and focus on proving the direct implication. This proof is done  by induction on the number of time steps. To avoid measurability issues in the induction step, we prove a slightly more general version of Theorem~\ref{thm:cpl_repr_timestep}.

\begin{theorem}\label{thm:cpl_repr_timestep2}
	Let $Z$ be a standard Borel space, $\mu$ be a kernel from $Z$ to $X$ and $\nu$ be a kernel from $Z$ to $Y$. Moreover, let $\pi$ be a kernel from $Z$ to $X \times Y$ such that $\pi^z \in \cplbc(\mu^z,\nu^z)$ for all $z \in Z$.
	
	Then there exists a Borel mapping $T : Z \times \widehat{X} \to \widehat{Y}$ such that for all $z \in Z$ the mappings $T^z : \widehat{X} \to \widehat{Y} : (x_{1:N},u_{1:N}) \mapsto T(z,x_{1:N},u_{1:N})$ are Borel isomorphisms satisfying 
	\begin{enumerate}[(i)]
		\item ${T^z}_\ast \widehat{\mu^z}=\widehat{\nu^z}$, or equivalently, $\widehat{\pi^z}:= (id,T^z)_\ast\widehat{\mu^z} \in \cplbc(\widehat{\mu^z},\widehat{\nu^z})$,
		\item ${\pr_{X\times Y}}_\ast\widehat{\pi^z}=\pi^z$.
	\end{enumerate}
\end{theorem}
\begin{proof}
	For one time step (i.e.\ $X=X_1$, $Y=Y_1$) bicausality is a trivial condition and biadapted is equivalent to bijective. Therefore, Theorem~\ref{thm:cpl_repr} is exactly the claim for one time step.
	
	Assume that we have already proven Theorem~\ref{thm:cpl_repr_timestep2} for $N-1$ time steps. Let $\mu$ be a kernel from $Z$ to $X_{1:N}$, $\nu$ be a kernel from $Z$ to $Y_{1:N}$ and $\pi$ be a kernel from $Z$ to $X_{1:N} \times Y_{1:N}$ satisfying $\pi^z \in \cplbc(\mu^z,\nu^z)$ for all $z \in Z$. 
	
	For each $z \in Z$ we can decompose $\mu^z,\nu^z$ and $\pi^z$ as
	\[
	d\mu^z=d\mu_1^zd\mu^{x_1,z}, \qquad d\nu^z=d\nu_1^zd\nu^{y_1,z}, \qquad d\pi^z=d\pi_1^zd\pi^{x_1,y_1,z}.
	\]
	and by Corollary \ref{cor:inductiveCharBicausal} it holds $
	\pi^{x_1,y_1,z} \in \cpl_{bc}(\mu^{x_1,z},\nu^{y_1,z}) 
	$. 
	
	By the induction hypothesis, there exists a Borel mapping
	\[
	S: (X_1\times Y_1 \times Z) \times \widehat{X}_{2:N} \to \widehat{Y}_{2:N}
	\]
	such that for all $(x_1,y_1,z) \in X_1\times Y_1 \times Z$ the mapping $S^{x_1,y_1,z}: X_{2:N} \to Y_{2_N} :(x_{2:N},y_{2:N}) \mapsto S(x_1,y_1,z,x_{2:N},y_{2:N})$ is a Borel isomorphism satisfying\footnote{As the notation given at the beginning of this section suggests, $\widehat{\mu^{x_1,z}} \in \prob(\widehat X_{2:N})$ is defined via
		\[
		\int f d\widehat{\mu^{x_1,z}} := \int f(x_{2:N},u_{2:N}) d\mu^{x_1,z}(x_{2:N})d\lambda^{N-1}(u_{2:N})
		\]
	}
	\begin{enumerate}[(i)]
		\item ${S^{x_1,y_1,z}}_\ast\widehat{\mu^{x_1,z}}=\widehat{ \nu^{y_1,z} }$,
		\item ${\pr_{XY}}_\ast\widehat{\pi^{x_1,y_1,z}} = \pi^{x_1,y_1,z}$, where $\widehat{\pi^{x_1,y_1,z}} := (\id,S^{x_1,y_1,z})_\ast \widehat{\mu^{x_1,z}}$.
	\end{enumerate}

	Moreover, by Theorem~\ref{thm:cpl_repr} there exists a Borel mapping
	\[
	T_1 : Z \times \widehat{X}_1 \to \widehat{Y}_1
	\]
	such that for all $z \in Z$, the mapping $T_1^z: \widehat{X}_1 \to \widehat{Y}_1 : (x_1,u_1) \mapsto T_1(z,x_1,u_1)$ is a Borel isomorphism satisfying
	${T^z_1}_\ast\widehat{\mu_1^z}=\widehat{\nu_1^z}$ and ${\pr_{XY}}_\ast(\id,T_1^z)_\ast\widehat{\mu_1^z}=\pi_1^z$.
	
	We define the mapping
	\[
	T: Z \times \widehat{X} \to \widehat{Y} : (z,x_{1:N},u_{1:N}) \mapsto (T_1^z(x_1,u_1),S^{x_1,\pr_{Y_1}(T_1^z(x_1,u_1))}(x_{2:N},u_{2:N}),z   )
	\]
	We have to check that $T$ has the desired properties:
	
	Clearly, $T$ is measurable as composition. Fix $z \in Z$. The mapping $T^z$ is biadapted by Lemma \ref{lem:indCharOfBiadapted}. 
	
	In order to check that $T_\ast\widehat{\mu^z}=\widehat{\nu^z}$, we consider an arbitrary Borel function  $f:\widehat{Y} \to \R$. We achieve by using the properties ${S^{x_1,y_1,z}}_\ast\widehat{ \mu^{x_1,z} }=\widehat{\nu^{y_1,z}}$ and ${T^z_1}_\ast\widehat{\mu_1^z}=\widehat{\nu_1^z}$
	\begin{align*}
		\int f(y_{1:N},v_{1:N})& dT^z_\ast\widehat{\mu^z}(y_{1:N},v_{1:N})= \\
		=&\int f(T_1^z(x_1,u_1),S^{x_1,\pr_{Y_1}(T_1^z(x_1,u_1)),z}(x_{2:N},u_{2:N})   )d\widehat{\mu^{x_1,z}}(x_{2:N},u_{2:N})d\widehat{\mu_1^z}(x_1,u_1)\\
		=&\int f(T^z_1(x_1,u_1),y_{2:N},v_{2:N})  \underbrace{d {S^{x_1,\pr_{Y_1}(T_1^z(x_1,u_1)),z}}_\ast\widehat{\mu^{x_1,z}}(y_{2:N},v_{2:N})}_{=d\reallywidehat{\nu^{\pr_{Y_1}\circ T^z_1(x_1,u_1),z}} (y_{2:N},v_{2:N}) })d\widehat{\mu_1^z}(x_1,u_1)\\
		=&\int f(y_{1:N},v_{1:N}) d\widehat{\nu^{y_1,z}}(y_{2:N},v_{2:N})d{T^z_1}_\ast\widehat{\mu_1^z}(y_1,v_1)\\
		=&\int f(y_{1:N},v_{1:N})d\widehat{\nu^z}(y_{1:N},v_{1:N}),
	\end{align*}
	i.e.\ $T^z_\ast\widehat{\mu^z}=\widehat{\nu}^z$.
	
	It remains to show that $\widehat{\pi^z} := (\id,T^z)_\ast\widehat{\mu^z}$ satisfies ${\pr_{XY}}_\ast\widehat{\pi^z}=\pi^z$. 
	Using ${\pr_{XY}}_\ast\widehat{\pi^{x_1,y_1,z}}=\pi^{x_1,y_1,z}$ and ${\pr_{XY}}_\ast (\id,T^z_1)_\ast\widehat{\mu_1^z}=\pi_1^z$ we obtain for any Borel function $f: X \times Y \to \R$:
	\begin{align*}
		\int f(x_{1:N},y_{1:N}) &d{\pr_{XY}}_\ast\widehat{\pi^z}(x_{1:N},y_{1:N}) =\\
		=&\int f (\pr_{XY} \circ (\id,T^z)(x_{1:N},u_{1:N}))d\widehat{\mu^{x_1,z}}(x_{2:N},u_{2:N})d\widehat{\mu_1^z}(x_1,u_1) \\
		=&\int f(\pr_{XY}((\id,T^z_1)(x_1,u_1)), \pr_{XY}((\id,S^{x_1,\pr_{Y_1}\circ T^z_1(x_1,u_1),z})(x_{2:N},u_{2:N})))\\[-0.2cm] & \qquad\qquad\qquad\qquad\qquad\, d\widehat{\mu^{x_1,z}}(x_{2:N},u_{2:N})d\widehat{\mu_1^z}(x_1,u_1)\\
		=&\int f(\pr_{XY}((\id,T_1)(x_1,u_1)),x_{2:N},y_{2:N})\\[-0.2cm] & \qquad\qquad\qquad\qquad\qquad
		d\pi^{x_1,\pr_{Y_1}\circ T^z_1(x_1,u_1),z}(x_{2:N},y_{2:N}) d\mu_1^z(x_1)d\lambda(u_1)\\
		=&\int f(x_{1:N},y_{1:N})d\pi^{x_1,y_1,z}(x_{2:N},y_{2:N}) d\pi^z_1(x_1,y_1)\\
		=& \int f(x_{1:N},y_{1:N}) d\pi(x_{1:N},y_{1:N}),
	\end{align*}
	which yields the desired result ${\pr_{XY}}_\ast\widehat{\pi^z} = \pi^z$.
\end{proof}

\begin{proof}[Proof of the reverse implication in Theorem~\ref{thm:cpl_repr_timestep}]
	We show a slightly more general claim: If $\widehat{\pi} \in \cplbc(\widehat{\mu},\widehat{\nu})$, then $\pi:= {\pr_{XY}}_\ast\widehat{\pi} \in \cplbc(\mu,\nu)$.  By Proposition~\ref{prop:bicausality_crit} it is enough to show that for every  $t<N$
	\begin{align}\label{eq:prf_conv_1}
		\widehat{\pi}^{x_{1:t},u_{1:t},y_{1:t},v_{1:t}} \in \cpl(\widehat{\mu}^{x_{1:t}, u_{1:t}} , \widehat{\nu}^{y_{1:t}, v_{1:t}}) \qquad \text{for $\widehat{\pi}_t$ almost all } (x_{1:t},u_{1:t},y_{1:t},v_{1:t})  
	\end{align}
	implies
	\begin{align}\label{eq:prf_conv_2}
		\pi^{x_{1:t},y_{1:t}} \in \cpl(\mu^{x_{1:t}} , \nu^{y_{1:t}}) \qquad \text{for $\pi_t$ almost all } (x_{1:t},y_{1:t}). 
	\end{align}
	We show that the condition on the first marginal in \eqref{eq:prf_conv_1} implies the condition on the first marginal in \eqref{eq:prf_conv_2}. The corresponding statement for the second marginal can be established in the same way. Note that 
	\begin{align}\label{eq:prf_conv_3}
		\widehat{\mu}^{x_{1:t},u_{1:t}} = \widehat{\mu^{x_{1:t}}} = \mu^{x_{1:t}} \otimes \lambda^t \quad \text{ and } \quad  \widehat{\pi}^{x_{1:t},u_{1:t},y_{1:t},v_{1:t}} = \widehat{\pi^{x_{1:t},y_{1:t}}} = \pi^{x_{1:t},y_{1:t}} \otimes \lambda^{2t}
	\end{align}
	for almost all $(u_{1:t},v_{1:t})$.
	Condition \eqref{eq:prf_conv_1} implies that we have for all $f_t : \widehat{X_{1:t}} \times \widehat{Y_{1:t}} \to \R $ and $g_t : \widehat{X_{t+1:N}} \to \R $ bounded measurable
	\begin{align*}
		&\int f_t(x_{1:t},u_{1:t},y_{1:t},v_{1:t}) g_t(x_{t+1:N},u_{x+1:N}) d\widehat{\pi}(x,u,y,v)=\\
		&\int f_t(x_{1:t},u_{1:t},y_{1:t},v_{1:t}) g_t(x_{t+1:N},u_{x+1:N}) d\widehat{\mu}^{x_{1:t},u_{1:t}}(x_{t+1:N},u_{x+1:N}) d\widehat{\pi_t}(x_{1:t}, u_{1:t},y_{1:t},v_{1:t})  
	\end{align*}
	Applying this fact to functions of the form  $f_t(x_{1:t},u_{1:t},y_{1:t},v_{1:t}) = f_t(x_{1:t},y_{1:t})$ and $g_t(x_{t+1:N},u_{t+1:N}) = g_t(x_{t+1:N})$ and invoking \eqref{eq:prf_conv_3} yields
	\begin{align*}
		\int f_t(x_{1:t},y_{1:t}) g_t(x_{t+1:N}) d\pi(x,y)=
		\int f_t(x_{1:t},y_{1:t}) g_t(x_{t+1:N}) d\mu^{x_{1:t}}(x_{t+1:N}) d\pi_t(x_{1:t}, y_{1:t}),  
	\end{align*}
	hence ${\pr_X}_\ast \pi^{x_{1:t},y_{1:t}} = \mu^{x_{1:t}}$ for $\pi_t$-almost all $x_{1:t}$.
\end{proof}

\subsection{Denseness of biadapted mappings in the set of bicausal couplings}\label{subsec:Denseness2}
Analogously to Section \ref{subsec:Denseness1} we use the representation result, Theorem~\ref{thm:cpl_repr_timestep}, to prove  denseness.

First, we state the regularity assumption on the marginals, which is essential for our proof: 
\begin{asn}\label{asn:cont_disint}
	Let $\mu \in \prob(X_{1:N})$. We say that $\mu$ satisfies Assumption \ref{asn:cont_disint} if $\mu$ has a disintegration
	\[
	d\mu(x_{1:N})=d\mu_1(x_1)d\mu^{x_1}(x_2)\cdots d\mu^{x_{1:N-1}}(x_N)
	\] 
	such that $\mu_1$ is continuous and, for all $t<N$ and $x_{1:t} \in X_{1:t}$, the measure $\mu^{x_{1:t}}$ is continuous.
\end{asn}
\begin{remark}
	Let $X_t = \R$ for $t \in \{1,\dots,N\}$. If $\mu \in \prob(X_{1:N})=\prob(\R^N)$ is absolutely continuous w.r.t.\  Lebesgue measure, it satisfies Assumption \ref{asn:cont_disint}. 
\end{remark}

In order to derive the denseness result from the representation result (Theorem~\ref{thm:cpl_repr_timestep}), we need to show that for every $\mu \in \mathcal{P}(X)$ there is an arbitraryily cheap biadapted map that pushes $\mu$ to $\widehat{\mu}$ (and that this construction can be chosen in a measureable way). As a first step, we strengthen the corresponding result for $N=1$ time periods from Section~\ref{sec:static_case}. 

\begin{lemma}\label{lem:basecase}
Let $\mu_1$ be a kernel from $Z$ to $X_1$ such that $\mu_1^z$ is continuous for every $z \in Z$. Let $S$ be a Polish space\footnote{We require that the metric $d_S$ induces the topology of $S$ but not that it is a complete metric.} and  $g : Z \times X_1 \to S $ be Borel and $\varepsilon >0$. Then there is a Borel map $\Phi : Z \times X_1 \to \widehat{X_1}$ such that for every $z \in Z$, the map $\Phi^z = \Phi(z,\cdot) : X_1 \to \widehat{X_1}$ is a Borel isomorphism satisfying $\Phi^z_\ast\mu_1^z=\widehat{\mu_1^z}$, $\int d_{X_1}(x_1, \pr_{X_1}(\Phi^z(x_1)))\wedge 1 \, \mu_1^z(dx_1) \le \varepsilon$ and $\int d_S(g^z(x_1), g^z(\pr_X(\Phi^z(x_1))))\wedge 1 \, \mu_1^z(dx_1) \le \varepsilon$.    
\end{lemma}
\begin{proof}

Let $ \mathcal{A}:=(A_i)_{i \in \N}$ be a partition of $X$ into Borel sets satisfying $\|\A\| \le \varepsilon$ and let $ \mathcal{B}:=(B_j)_{j \in \N}$ be a partition of $S$ into Borel sets satisfying $\|\B\| \le \varepsilon$. We set 
$
M^{i,j}_z := A_i \cap (g^z)^{-1}(B_j) 
$.
Note that for $x,x' \in M^{i,j}_z$ we have $d_X(x,x')\le \varepsilon$ and $d_S(g^z(x),g^z(x')) \le \varepsilon$.  Hence, applying Proposition~\ref{prop:part_pushfwd2} yields the claim.
\end{proof}


\begin{prop}\label{prop:part_pushfwd_adapted2}
	Let $Z$ be a Polish space and, let $\mu$ be a kernel from $Z$ to $X$ s.t.\ $\mu^z$ satisfies Assumption \ref{asn:cont_disint} for all $z \in Z$, and let $\varepsilon>0$. Then there exists a Borel mapping $\Phi : Z \times X \to \widehat{X}$ s.t.\ for all $z \in Z$ the mapping $\Phi^{z} = \Phi(z,\cdot ) : X \to \widehat{X}$ is  biadapted and satisfies $\Phi^{z}_\ast\mu^z = \widehat{\mu^z}$ and $\int d_X(x,\pr_X(\Phi^z(x))) \wedge 1\, d\mu^z(x) \le \varepsilon$.
\end{prop}

Next, we state the denseness result, extending Theorem~\ref{MainTheorem} from the introduction.

\begin{theorem}\label{thm:main}
	Let $\mu \in \prob(X)$ and $\nu \in \prob(Y)$ satisfy Assumption \ref{asn:cont_disint}. Then the set of biadapted Monge couplings between $\mu$ and $\nu$ is weakly dense in  $\cplbc(\mu,\nu)$, i.e.
	$$\overline{\cplbc(\mu,\nu) \cap \cplbm(\mu,\nu)} = \cplbc(\mu,\nu). $$
	Let $p \in [1,\infty)$. If $\mu$ and $\nu$ have finite $p$-th moments (w.r.t. compatible metrics $d_X, d_Y$), we have $\W_p$-denseness as well.	

    Moreover, the approximating sequences can be chosen in a measurable way: Given a Polish space $Z$ and a kernel $\pi$ from $Z$ to $X \times Y$ such that  $ \mu^z := {\pr_X}_\ast \pi^z$ and  $\nu^z:= {\pr_Y}_\ast \pi^z$ satisfy Assumption~\ref{asn:cont_disint} and $\pi^z \in \cplbc(\mu^z,\nu^z)$ for every $z$, there exist Borel maps $T_n : Z \times X  \to  Y$ such that for every $z \in Z$, $T_n^z : X \to Y$ is a biadapted map with ${T_n^z}_\ast \mu^z = \nu^z $ and $(\id,T_n^z)_\ast\mu^z \weaklyto \pi^z$. 
\end{theorem}

We  prove Proposition~\ref{prop:part_pushfwd_adapted2} and Theorem~\ref{thm:main} simultaneously by induction on the number of time steps $N$. Specifically, we will show that  Theorem~\ref{thm:main} for $N-1$ time periods implies Proposition~\ref{prop:part_pushfwd_adapted2} for $N$ time periods and that Proposition~\ref{prop:part_pushfwd_adapted2} for $N$ time periods implies Theorem~\ref{thm:main} for $N$ time periods. Note that Lemma~\ref{lem:basecase} serves as base case for this induction as it is clearly stronger than Proposition~\ref{prop:part_pushfwd_adapted2} for $N=1$.


\begin{proof}[Proof of Proposition~\ref{prop:part_pushfwd_adapted2} for $N$ time periods using Theorem~\ref{thm:main} for $N-1$ time periods]$\quad$\\ Fix $\epsilon>0$. We further disintegrate the kernel $\mu$ to
\[
d\mu^z(x_{1:N}) = d\mu_1^z(x_1)d\mu^{z,x_1}(x_{2:N})
\]
and note that the disintegration can be chosen such that $\mu^{z,x_1}$ satisfies Assumption \ref{asn:cont_disint} for every $x_1 \in X_1$ and $z \in Z$. 

By Lemma~\ref{lem:basecase} applied with $(S,d_S)=(\mathcal{P}(X_{2:N}),\AW)$ where $\AW$ is the adapted Wasserstein distance w.r.t.\ the metric $d_{X_{2:N}}\wedge 1$ and $g(z,x_1) = \mu^{z,x_1}$, there is a Borel map 
$$
\Phi_{1} : Z \times X_1 \to \widehat{X_1} 
$$
s.t.\ for all $z \in Z$ the mapping $\Phi_{1}^z=\Phi_1(z,\cdot) : X_1  \to \widehat{X_1}$ is a Borel isomorphism satisfying  ${\Phi_{1}^z}_\ast\mu_1^z=\widehat{\mu_1^z}$ and 
\begin{align}\label{eq:phi1_distortion}
\int d_{X_1}(x_1, \pr_{X_1}(\Phi_{1}^z(x_1))) \wedge 1 \mu_1^z(dx_1) \le \varepsilon/3, \quad \int \AW(\mu^{z,x_1}, \mu^{z,\Phi_{1}^z(x_1)}) d\mu_1^z(x_1) \le \varepsilon/3 .
\end{align}

Next, we derive from Theorem~\ref{thm:main} for $N-1$ time periods that there is a measurable mapping 
	\begin{align*}
	\Psi : (Z \times X_1) \times X_{2:N} \to \widehat{X}_{2:N}
	\end{align*}
	s.t.\ for all $(z,x_1) \in Z \times X_1$ the mapping $\Psi^{z,x_1}=\Psi(z,x_1,\cdot) :  X_{2:N} \to \widehat{X}_{2:N}$ is a biadapted mapping satisfying ${\Psi^{z,x_1}}_\ast\mu^{z,x_1}=\reallywidehat{\mu^{z,\Phi^z_{1}(x_1)}}$ and 
    \begin{align}\label{eq:almostOpt}
    \int d_{X_{2:N}}(x_{2:N}, \pr_{X_{2:N}}(  \Psi^{z,x_1}( x_{2:N} )   ) ) \wedge 1 \,  d\mu_1^{z,x_1}(x_{2:N}) \le \AW(\mu^{z,x_1}, \mu^{z,\Phi_{1}^z(x_1)}) + \varepsilon/3.
    \end{align}
    Indeed, such a map exists because $\AW(\mu^{z,x_1}, \mu^{z,\Phi_{1}^z(x_1)})$ is precisely given by the minimum of $\int d_{X_{2:N}}(x_{2:N}, y_{2:N}) \wedge 1 d\pi(x_{2:N}, y_{2:N}, v_{2:N}  )$ among all $\pi \in \cplbc(\mu^{z,x_1}, \reallywidehat{ \mu^{z,\Phi_{1}^z(x_1)} } )$ and Theorem~\ref{thm:main} for $N-1$ time periods guarantees the existence of an $\varepsilon/3$-optimal biadapted map for this bicausal transport problem (and its measurable dependence on the parameter $z$). 
	
	Next, we define the mapping 
	\[
	\Phi : Z \times X \to \widehat{X} : (z,x_{1:N}) \mapsto (\Phi_{1}^z(x_1),\Psi^{z,x_1}(x_{2:N})).
	\]
	Clearly, $\Phi$ is measurable as concatenation and for all $z \in Z$ the map $\Phi^z$ is biadapted by Lemma~\ref{lem:indCharOfBiadapted}. Next, we observe that ${\Phi^z}_\ast \mu^z = \widehat{\mu^z}$. Indeed, for all bounded Borel $f : \widehat{X} \to \R$ we find
	\begin{align*}
		&\int f(x,u)d{\Phi^{z}}_\ast\mu^z(x,u) = \\
		&=\iint f(\Phi_{1}^z(x_1),\Psi^{z,x_1}(x_{2:N})) d\mu^{z,x_1}(x_{2:N})d\mu_1^z(x_1)\\
		&=\iint f(\Phi_{1}^z(x_1),x_{2:N},u_{2:N}) \underbrace{d {\Psi^{z,x_1}}_\ast\mu^{z,x_1}(x_{2:N},u_{2:N})}_{  =d\reallywidehat{\mu^{z,\Phi^z_{1}(x_1)}}(x_{2:N},u_{2:N})  }d\mu_1^z(x_1)\\
		&=\iint f(x_1,u_1,x_{2:N},u_{2:N}) d\reallywidehat{\mu^{z,x_1}}(x_{2:N},u_{2:N}) d{\Phi_1^z}_\ast\mu_1^z(x_1,u_1)\\
		&=\int f(x,u)d\reallywidehat{\mu^z}(x,u),
	\end{align*}
	which yields ${\Phi^{z}}_\ast\mu^z = \reallywidehat{\mu^z} $. Using \eqref{eq:phi1_distortion} and \eqref{eq:almostOpt}, we estimate
\begin{align*}
    \int d_X(x,&\pr_X(\Phi^z(x))) \wedge 1\, d\mu^z(x) \\
    &\le \int d_{X_1}(x_1, \pr_{X_1}(\Phi^z_{1}(x_1)) \wedge 1 \, d\mu^z_1(x_1)  \\
    &\qquad\qquad+\iint d_{X_{2:N}}(x_{2:N}, \pr_{X_{2:N}}(\Psi^{z,x_1}(x_{2:N}))) \wedge 1 \, d\mu^{z,x_1}(x_{2:N}) d\mu_1^z(x_1) \\
    &\le \frac{2\varepsilon}{3} + \int \AW(\mu^{z,x_1}, \mu^{z,\Phi^z_{1}(x_1)}) \, d\mu_1^z(x_1) \le \varepsilon. \qedhere
\end{align*}
\end{proof}

\begin{proof}[Proof of Theorem~\ref{thm:main} for $N$ time periods using Proposition~\ref{prop:part_pushfwd_adapted2} for $N$ time periods]
$\quad$\\
Let $\pi$ be a kernel from $Z$ to $X \times Y$ such that  $ \mu^z := {\pr_X}_\ast \pi^z$ and  $\nu^z:= {\pr_Y}_\ast \pi^z$ satisfy Assumption~\ref{asn:cont_disint} and $\pi^z \in \cplbc(\mu^z,\nu^z)$ for every $z \in Z$. By  Theorem~\ref{thm:cpl_repr_timestep2}, there is a Borel mapping $T : Z \times \widehat{X} \to \widehat{Y}$ such that for all $z \in Z$ the mappings $T^z : \widehat{X} \to \widehat{Y} : (x_{1:N},u_{1:N}) \mapsto T(z,x_{1:N},u_{1:N})$ are Borel isomorphisms satisfying 
	\begin{enumerate}[(i)]
		\item ${T^z}_\ast \widehat{\mu^z}=\widehat{\nu^z}$, or equivalently, $\widehat{\pi^z}:= (id,T^z)_\ast\widehat{\mu^z} \in \cplbc(\widehat{\mu^z},\widehat{\nu^z})$,
		\item ${\pr_{X\times Y}}_\ast\widehat{\pi^z}=\pi^z$.
	\end{enumerate}
By Proposition~\ref{prop:part_pushfwd_adapted2}, there are Borel mappings $\Phi_n : Z \times X \to \widehat{X}$ and $\Psi_n : Z \times Y \to \widehat{Y}$ such that
\begin{enumerate}[(i)]
		\setcounter{enumi}{2}
		\item ${\Phi^z_n}_\ast\mu^z = \widehat{\mu^z}$ for every $z \in Z$,
		\item${\Psi^z_n}_\ast\nu^z = \widehat{\nu^z}$ for every $z \in Z$,
	\end{enumerate}
    and we have
\begin{enumerate}[(i)]
		\setcounter{enumi}{4}
		\item $\int d_X(x, \pr_X(\Phi_n^z(x))) \wedge 1 \, \mu^z(dx) \to 0$ for every $z \in Z$,
		\item$\int d_Y(y, \pr_Y(\Psi_n^z(y))) \wedge 1 \, \nu^z(dy) \to 0$ for every $z \in Z$.
	\end{enumerate}

We define the desired maps $T_n$ by setting for every $z \in Z$ and $n \in \N$, 
\[
	T^z_n := {(\Psi^z_n)}^{-1} \circ T^z \circ \Phi^z_n : X \to Y .
	\]
It is straightforward to see that $T_n$ is Borel and that for every $z \in Z$, $T_n^z$ is a biadapted map satisfying ${{T_n^z}_\ast}\mu^z=\nu^z$. 

It remains to show that $(\id,T_n^z)_\ast\mu^z \weaklyto \pi^z$ for every $z \in Z$. To this end, it suffices to show that for every $f : X \times Y \to \R$ that is 1-Lipschitz w.r.t.\ $(d_X+d_Y)\wedge 1$ we have $\int f(x,T^z_n(x))\, d\mu^z(x) \to \int f(x,y)\, d\pi^z(x,y)$. Indeed, if $f$ is such a 1-Lipschitz function, using that $T^z_n = {(\Psi^z_n)}^{-1} \circ T^z \circ \Phi^z_n$ and ${\Phi_n^z}_\ast \mu^z = \widehat{\mu^z}$, we estimate
\begin{align*}
    &\left| \int f(x,T_n^z(x)) \, d\mu^z(x) - \int f(x, ({\Psi_n^z})^{-1}(T(x,u))) d\widehat{\mu^z}(x,u)  \right|  = \\   
    &= \left| \int  f(x,T_n^z(x)) -  f(\pr_X(\Phi_n^z(x)),T_n^z(x))     d\mu^z(x) \right| \\
    &\le \int d_X(x, \pr_X(\Phi_n^z(x))) d\mu^z(x) \to 0.
\end{align*}
Moreover,  using that $\widehat{\pi^z} = (\id,T^z)_\ast\mu^z$ and ${\Psi^z_n}_\ast \nu^z = \widehat{\nu^z}$, we estimate
\begin{align*}
    &\left| \int f(x, ({\Psi_n^z})^{-1}(T(x,u))) d\widehat{\mu^z}(x,u) - \int f(x,y) d\pi^z(x,y)  \right| = \\
    &= \left|  \int f(x, (\Psi^z_n)^{-1}(y,v)) - f(x,y) d\widehat{\pi^z}(x,u,y,v)  \right| \\
    &\le \int d_Y(({\Psi_n^z})^{-1}(y,v), y) d\widehat{\nu^z}(y,v) = \int d_Y(y, \pr_Y(\Phi_n^z(y))) d\nu^z(y) \to 0.
\end{align*}
From these two estimates, we obtain $| \int f(x,T^z_n(x)) d\mu^z(x) - \int f(x,y)d\pi^z(x,y)| \to 0 $. 

If $\mu^z,\nu^z$ have finite $p$-th moments, for fixed $(x_0,y_0) \in X \times Y$, we have for every $n \in \N$
$$
\int d_{X \times Y}((x,T^z_n(x)),(x_0,y_0))d\mu^z(x) = \int d_X(x,x_0)d\mu^z(x) + \int d_Y(y, y_0) d\nu^z(y).
$$
Hence, $(\id,T_n^z)_\ast\mu^z \weaklyto \pi^z$  implies that $\W_p( (\id,T_n^z)_\ast\mu^z, \pi^z) \to 0$, see e.g.\ \cite[Theorem~7.12]{Vi03}. 
\end{proof}

\subsection{Discussion}
In \cite{BeLa20} a version of the result for causal couplings was shown, i.e.\ causal couplings supported on the graph of adapted mappings are dense in the set of causal couplings with fixed marginals. It was sufficient to require continuity assumptions for the first times tep of the $X$-marginal.

Since bicausality is causality in both directions, it is clear that we will need at least the same assumption for the $Y$-marginal as well. However, the assumptions in Theorem~\ref{thm:main} are significantly stronger than that. We give an example to illustrate why more restrictive assumptions than continuity of the $X$- and $Y$- marginals in the first time step are necessary.

\begin{example}\label{AbsolutelyImportant}
	Let $N=2$ and $X_1=X_2=Y_1=Y_2=[0,1]$. Consider the measures $\mu= \lambda \otimes \delta_0 \in \prob(X_1 \times X_2)$ and $\nu= \lambda^2 \in \prob(Y_1 \times Y_2)$. 
	
	Since $\mu$ and $\nu$ are both continuous measures, couplings between $\mu$ and $\nu$ supported on the graphs of bijections are dense in $\cpl(\mu,\nu)$, see Theorem~\ref{thm:cplbm_dense}. We show that there are no bicausal couplings between $\mu$ and $\nu$ that are supported on the graph of a bijection.
	
	Assume that there exists  $\pi \in \cplbc(\mu,\nu) \cap \cplbm(\mu,\nu)$ and decompose it as $d\pi(x_1,x_2,y_1,y_2)=d\pi_1(x_1,y_1)d\pi^{x_1,y_1}(x_2,y_2)$. By Corollary \ref{cor:inductiveCharBicausal} we get $\pi^{x_1,y_1} \in \cplbm(\mu^{x_1},\nu^{y_1})$ for $\pi_1$-almost all $(x_1,y_1)$. However, for all $x_1, y_1$ it holds $\cplbm(\mu^{x_1},\nu^{y_1}) =  \cplbm(\delta_0,\lambda)=\emptyset$. Therefore, such a $\pi$ cannot exist and $\cplbc(\mu,\nu) \cap \cplbm(\mu,\nu)$ is empty.
	
	
\end{example}

To close this section, we state the consequences of Theorem~\ref{thm:main} for the bicausal transport problem and  the adapted Wasserstein distance.

\begin{cor}
	Let $\mu \in \prob(X)$ and $\nu \in \prob(Y)$ satisfy Assumption \ref{asn:cont_disint} and $c : X \times Y \to \R$ be continuous and bounded. Then we have
	\[
	\inf \left\{ \int c d\pi \:\middle|\:  \pi \in \cpl_{bc}(\mu,\nu)   \right\} = \inf \left\{ \int c(x,T(x)) d\mu(x) \:\middle|\:  T \text{ biadapted, } T_\ast \mu=\nu \right\}.
	\]
\end{cor}

In particular, one can restrict to couplings supported on the graph of biadapted mappings, when calculating the adapted Wasserstein distance of two probability measures:
\begin{cor}
	Let $p \in [1,\infty)$ and equipp the Polish space $X$ with the product metric $d$ of the metrics $d_1,\dots,d_N$ on the spaces $X_1,\dots,X_N$. For $\mu,\nu \in \prob_p(X)$ satisfying Assumption \ref{asn:cont_disint} we have   \[
	\AW_p(\mu,\nu) = \inf \left\{ \int d(x,T(x))^p \mu(dx)  \:\middle|\: T \text{ biadapted, } T_\ast\mu=\nu    \right\}^{1/p}.
	\]
\end{cor}

\appendix

	\section{Appendix}

	\subsection{Preliminaries from descriptive set theory}\label{APrelimSec}
	In this section we briefly state a few results from descriptive set theory that we used often throughout this paper, for more details the reader is referred to the monograph \cite{Ke95}.
	
	A measurable space is a tuple $(X,\A)$, where $X$ is a set and $\A$ is $\sigma$-algebra on $X$. A measurable space $(X,\A)$ is called standard Borel, if there exists a Polish topology $\T$ on $X$ such that the $\sigma$-algebra $\A$ is the Borel $\sigma$-algebra generated by $\T$. We write $X$ instead of $(X,\A)$, if the $\sigma$-algebra is clear from the context.
	
	Let $X$, $Y$ be standard Borel spaces and $f: X \to Y$ a mapping. Then $f$ is Borel if and only if $\graph(f)$ is Borel (see {\cite[Theorem~14.12]{Ke95}}).
	In particular, if $f : X \to Y$ is Borel and  bijective, then $f$ is a Borel isomorphism (i.e.\ $f^{-1}$ is measurable). 

	Let $X, Y$ be standard Borel spaces. Then $X$ and $Y$ are Borel isomorphic if and only if they have the same cardinality (``Borel isomorphism Theorem'', see {\cite[Theorem~15.6]{Ke95}}).
	All uncountable standard Borel spaces have the cardinality of the continuum (see \cite[Theorem~13.6]{Ke95}) and are therefore Borel isomorphic.
	
	Recall that a probability measure $\mu$ is called continuous if it does not give mass to singletons, i.e. $\mu(\{x\})=0$ for all $x \in X$.	Let $X$ be a standard Borel space and $\mu \in \prob(X)$ be continuous. Then there is a Borel isomorphism $f : X \to [0,1]$ s.t.\ $f_\ast\mu=\lambda$, where $\lambda$ denotes the Lebesgue measure on $[0,1]$ (``Isomorphism theorem for measures'', see \cite[Theorem~17.41]{Ke95}). In particular, every standard Borel space that supports a continuous probability measure has the cardinality of the continuum.

	%
	%
	%
	%


	\subsection{Omitted proofs}\label{subsec:proofs} For the sake of completeness we state the proofs of a few technical lemmas that we have omitted as they are not crucial for the understanding of the main results.
	
	Recall that we assumed all standard Borel spaces to be uncountable.
	\begin{lemma}\label{lem:uncountableNullset}
		Let $X$ be a standard Borel space and $\mu \in \prob(X)$. Then there is an uncountable Borel set $A \subset X$ satisfying $\mu(A)=0$. 
	\end{lemma}
	\begin{proof}
		Let $X' := \{ x \in X \big| \mu(\{x\}) =0 \}$. Since $X \setminus X'$ is countable, $X'$ is uncountable Borel. If $\mu(X')=0$, take $A:=X'$. Otherwise let $\nu:= \frac{1}{\mu(X')}\mu|_{X'} \in \prob(X')$. By the isomorphism  theorem for measures (c.f.\ Section~\ref{APrelimSec}) there is a Borel isomorphism $f : X' \to [0,1]$ s.t.\ $f_\ast\nu=\lambda$. Denote $C \subset [0,1]$ the usual Cantor set and let $A:= f^{-1}(C)$. 	
	\end{proof}
	%
	%
	
	\begin{proof}[Proof of Lemma \ref{lem:indCharOfBiadapted}]
		It is clear that $T$ is adapted and it is easy to check that it is injective and surjective. Therefore,  $F:=T^{-1}$ exists and it suffices to show that $F$ is adapted. To that end, we denote
		\[
		F(y_{1:N})=(F_1(y_{1:N}),F_{2:N}(y_{1:N})).
		\]
		We have
		\begin{align}\label{eq:prf:biad_induc}
			y_{1:N}=T(F(y_{1:N})) = (T_1(F_1(y_{1:N})),S^{F_1(y_{1:N})}(F_{2:N}(y_{1:N})))
		\end{align}
		and therefore $y_1=T_1(F_1(y_{1:N}))$. As $T_1$ is bijective, we can apply $T_1^{-1}$ to get $F_1(y_{1:N})=T_1^{-1}(y_1)$, i.e.  $F_1$ depends only on $y_1$. Moreover, \eqref{eq:prf:biad_induc} implies $y_{2:N}=S^{F_1(y_1)}(F_{2:N}(y_{1:N}))$. Since $S^{F_1(y_1)}$ is bijective, this implies
		\[
		(S^{F_1(y_1)})^{-1}(y_{2:N}) = F_{2:N}(y_{1:N}).
		\]
		Since $S^{F_1(y_1)}$ is assumed to be biadapted, there exists, for every $t$, $F_t : Y_{1:t} \to X_t$ Borel s.t.\
		\[
		F_{2:N}(y_{1:N})=(F_2(y_{1:2}),\dots,F_N(y_{1:N})).\qedhere
		\]
	\end{proof}
	
	\begin{lemma}\label{lem:causal_monge_cpl} For $\mu \in \prob(X)$ and $\nu \in \prob(Y)$ we have
		\begin{enumerate}[(a)]
			\item $\cplc(\mu,\nu) \cap \cplm(\mu,\nu) = \{ (id,T)_\ast \mu \big| T \text{ adapted }, T_\ast\mu=\nu  \}$
			\item $\cplbc(\mu,\nu) \cap \cplbm(\mu,\nu) = \{ (id,T)_\ast \mu \big| T \text{ biadapted }, T_\ast\mu=\nu  \}$
		\end{enumerate}
		
	\end{lemma}
	
	\begin{proof}
		(a) Let $\pi \in \cplc(\mu,\nu) \cap \cplm(\mu,\nu)$, write $d\pi=d\delta_{T(x)}d\mu(x)$ and fix $t \le N$. Then for all $B \in \F_t^Y$, the mapping
		\[
		\phi_B : x \mapsto \delta_{T(x)}(B)
		\]
		is $\F_t^X$ measurable. Since $T^{-1}(B) = \phi_B^{-1}(\{1\})$, this implies that $T$ is $\F_t^X$-$\F_t^Y$-measurable. Hence, $\pr_{Y_{1:t}} \circ T$ is $\F_t^X$-measurable. By the Doob–Dynkin lemma this implies that there is a Borel function $\widetilde T_t : X_{1:t} \to Y_{1:t}$ s.t.\ $\pr_{Y_{1:t}} \circ T = \widetilde{T}_t \circ \pr_{X_{1:t}}$. Set $T_t := \pr_{Y_t} \circ \widetilde{T}_t$.
		
		Conversely, let $d\pi=d\delta_{T(x)}d\mu(x)$, where $T$ is adapted and satisfies $T_\ast \mu=\nu$. For $t<N$ and $B \in \F_t^Y$ we need to show that $\phi_B$ is $\F_t^X$-measurable. Since $\phi_B$ only takes the values 0 and 1, it suffices to show that $\phi_B^{-1}(\{1\}) \in \F_t^X$. Indeed, $B = \pr_{Y_{1:t}}^{-1}(B')$ for some Borel $B' \subset Y_{1:t}$ and
		$
		\phi_B^{-1}(\{1\}) = T^{-1}(B) =  (\pr_{Y_{1:t}} \circ T)^{-1}(B') \in \F_t^X,
		$
		since $\pr_{Y_{1:t}} \circ T$ is $\F_t^X$-measurable.
		
		(b) is an immediate consequence of (a).
	\end{proof}

	\subsection{An isomorphism theorem for kernels}
	The isomorphism theorem for measures  states that for any continuous measure on a standard Borel space $X$, there exists a bijection $f: X \to [0,1]$ s.t.\ $f_\ast\mu=\lambda$. The main goal of this section is to prove the following  parameterized version: 
	
	\begin{theorem}\label{thm:borel_iso_parametr}
		Let $X$ and $Y$ be standard Borel spaces and $\pi$ a kernel from $X$ to $Y$ s.t.\ $\pi^x$ is a continuous probability measure for all $x \in X$. Then there exists a Borel function 
		\[
		G : X \times Y \to [0,1]
		\] 
		such that for all $x \in X$ the mapping $G^x = G(x, \cdot ) : Y \to [0,1]$ is a Borel ismorphism satisfying $G^x_\ast\pi^x = \lambda$, where $\lambda$ denotes the Lebesgue measure on $[0,1]$.
	\end{theorem}
	
	
	Using the axiom of choice, one could choose for all $x \in X$ a Borel isomorphism $f_x$ that pushes $\pi^x$ to $\lambda$ and define $f(x,y):=(x,f_x(y))$. However, there is no reason why this function $f$ is measurable. Therefore, we repeat the construction  in the proof of the isomorphism theorem for measures given in {\cite[Theorem~17.41]{Ke95}} in a way that is uniform for all $x$. A key role for this plays Theorem~2.4 from \cite{Ma79} because it ensures the existence of Borel isomorphisms that let the $x$-coordinate fixed under suitable conditions. We will first clarify what this means exactly:
	
	For a set $A \subset X \times Y$ and $x \in X$ we define the $x$-section of $A$ as $A_x :=\{y \in Y|(x,y) \in A\}$. 
	
	\begin{definition}
		Let $X, Y, Z$ be standard Borel spaces and let $B \subset X \times Y$ be Borel. A \emph{Borel parametrization} of $B$ is a Borel isomorphism $f : X \times Z \to B$ satisfying  $f(\{x\} \times Z ) = \{x\} \times B_x$ for all $x \in X$.
	\end{definition}
	
	Of course, a necessary condition for the existence of a Borel parametrization is that all $x$-sections of $B$ have the same cardinality. We are interested in the case where $X$ is uncountable and all $x$-sections of $B$ are uncountable. The following theorem gives a criterion for the existence of Borel parametrizations:
	
	\begin{theorem}[{\cite[Theorem~2.4]{Ma79}}]\label{thm:Borel_param}
		Let $X$ and $Y$ be uncountable standard Borel spaces and let $B \subset X \times Y$ be a Borel set with uncountable $x$-sections. Then the following are equivalent
		\begin{enumerate}[(i)]
			\item $B$ has a Borel parametrization.
			\item There is a Borel set $M \subset B$ such that for all $x \in X$ the set $M_x$ is compact and perfect.
			\item There exists a kernel $\mu$  from $X$ to $Y$ such that for all $x \in X$ the measure $\mu^x$ is continuous and satisfies $\mu^x(B_x)=1$.
		\end{enumerate}
	\end{theorem}
	
	
	\begin{remark}
		If $B \subset X \times Y$ is a Borel set with uncountable $x$-sections, all $x$-sections contain a homeomorphic copy of $\{0,1\}^\N$ and therefore a compact perfect set. Loosely speaking, the assertion (ii) in Theorem~\ref{thm:Borel_param} says that these perfect sets can be chosen in a uniform way.
	\end{remark}
	
	Given a kernel $\pi$ from $X$ to $Y=[0,1]$ such that $\pi^x$ is continuous for all $x \in X$, the function $G(x,t):= F_{\pi^x}(t)$ is jointly measurable and $G^x=G(x,\cdot)$ pushes $\pi^x$ to $\lambda$. However, $G^x$ is in general not injective: If $\supp(\pi^x) \subsetneq [0,1]$, the function $F_{\pi^x}$ is constant on non-trivial intervals. The following lemma asserts that the set, where this problem occurs, is Borel in the product space:
	\begin{lemma}\label{lem:set_const_Borel}
		Let $F : X \times [0,1] \to [0,1]$ be a Borel mapping such that for all $x \in [0,1]$ the mapping $t \mapsto F(x,t)$ is monotone. Then the set
		\[
		N:=\{ (x,y) \in X \times [0,1] \:|\: \exists t_1 \neq t_2 : y= F(x,t_1)=F(x,t_2) \}
		\]
		is Borel and $N_x$ is at most countable for all $x \in X$.
	\end{lemma}
	\begin{proof} 
		Since $t \mapsto F(x,t)$ is monotone, we note that
		\[
		(x,y) \in N \iff \exists t_1,t_2 \in \mathbb Q \cap [0,1], t_1 \neq t_2 \text{ s.t.\ } y = F(x,t_1) = F(x,t_2).
		\]
		Hence, we can express $N$ as a countable union of intersections of graphs of Borel functions: 
		\begin{equation} \label{eq:lem.set_const_Borel}
			N = \bigcup_{t_1, t_2 \in \mathbb{Q} \cap [0,1], t_1 \neq t_2} \left\{ (x,F(x,t_1)) : x \in X \right\} \cap \left\{ (x,F(x,t_2)) : x \in X \right\}.
		\end{equation}
		As graphs of Borel functions are Borel measurable (see Appendix~\ref{APrelimSec}), we have shown that $N$ is Borel. 
		Finally, we find by \eqref{eq:lem.set_const_Borel} that $N_x$ is a (at most) countable union of singletons, and hence countable.
	\end{proof}

	\begin{prop}\label{prop:perfect_nullset}
		Let $X$ be a Polish space and $\pi$ a kernel from $X$ to $[0,1]$ s.t.\ $\pi^x$ is continuous for all $x \in X$. Then there is a Borel set $M \subset X \times [0,1]$ such that for all $x \in X$ the set $M_x$ is compact perfect and satisfies $\pi^x(M_x)=0$.
	\end{prop}
	\begin{proof}
		Let $(U_n)_{n \in \N}$ be a base of the standard topology of $[0,1]$. Each $U_n$ contains a non-empty open interval, hence an non-degenerate closed interval $C_n$. Clearly, $C_n$ is compact perfect. The usual Cantor set $C \subset [0,1]$ is compact perfect as well. Define a set $M \subset X \times [0,1]$ via	
		\begin{align*}
			(x,t) \in M :\iff \begin{cases}
				t \in F_{\pi^x}(C) & \text{if } \supp(\pi^x)=[0,1]\\
				t \in C_n & \text{if } n= \min\{k| \pi^x(U_k)=0\}.
			\end{cases}
		\end{align*}
		The mapping $(x,t) \mapsto F_{\pi^x}(t)$ is measurable in $x$ and continuous in $t$, hence jointly measurable (see e.g.\  \cite[Lemma 4.51]{AlBo07}).  Therefore, $\{(x,t)| t \in F_{\pi^x}(C)\}$ is Borel and hence $M$ is Borel.
		
		If $\supp(\pi^x)=[0,1]$ the function $F_{\pi^x}$ is a homeomorphism, hence $M_x =  F_{\pi^x}(C)$ is compact perfect and $\pi^x(M_x)=\lambda(C)=0$. If $\supp(\pi^x) \neq[0,1]$, there is some (and therefore a minimal) $n$ s.t.\ $\pi^x(U_n)=0$ and we have $M_x=C_n$ and hence $\pi^x(M_x) \le \pi^x(U_n)=0$.
	\end{proof}
	
	\begin{proof}[Proof of Theorem \ref{thm:borel_iso_parametr}]
		By the Borel isomorphism Theorem~(c.f.\ Section~\ref{APrelimSec}) we can assume that $Y=[0,1]$. By \cite[Lemma 4.51]{AlBo07} the mapping
		\[
		F: X \times [0,1] \to X \times [0,1] : (x,t) \mapsto (x,F_{\pi^x}(t))
		\]
		is jointly measurable. By Lemma \ref{lem:set_const_Borel} the set 
		\[
		N:=\{ (x,y) \in X \times [0,1] \:|\: \exists t_1 \neq t_2 : y= F(x,t_1)=F(x,t_2) \}
		\]
		is Borel and $N_x$ is at most countable for all $x \in X$. The set $M:= F^{-1}(N)$ is Borel and satisfies $\pi^x(M_x)=\pi^x([F(x,\cdot)]^{-1}(N_x))= \lambda(N_x)=0$ for all $x \in X$ because $N_x$ is countable. Clearly, $F$ is a bijection between $(X \times [0,1]) \setminus M$ and $(X \times [0,1]) \setminus N$.
		
		By Proposition~\ref{prop:perfect_nullset} there exist Borel sets $A, B \subset X \times [0,1]$ such that for all $x \in X$ the sets $A_x$ and $B_x$ are compact perfect and $\pi^x(A_x)=0$ and $\lambda(B_x)=0$. This implies that $\lambda([F(x,\cdot)](A_x))=0$ and $\pi^x([F(x,\cdot)]^{-1}(B_x))=0$.
		
		Consider the sets $C:= A \cup F^{-1}(B) \cup M$ and $D:= F(A)\cup B \cup N$. Then we have $\pi^x(C_x)\le \pi^x(A_x)+ \pi^x([F(x,\cdot)]^{-1}(B_x))+ \pi^x(M_x)=0$ and $\lambda(D_x) \le \lambda([F(x,\cdot)](A_x)) + \lambda(B_x) + \lambda(N_x)=0$ for all $x \in X$ and $F$ is a bijection between $(X \times [0,1]) \setminus C$ and $(X \times [0,1]) \setminus D$. Moreover, $C$ and $D$ both satisfy the assumptions of Theorem~\ref{thm:Borel_param}, so there exist Borel parametrizations $f : X \times [0,1] \to C$ and $g: X \times [0,1] \to D$. 
		
		Clearly, the mapping 
		\[
		\widetilde G : X \times [0,1] \to  X \times [0,1] : (x,t) \mapsto \begin{cases}
			F(x,t) & t \in [0,1] \setminus C_x,\\
			g(f^{-1}(x,t))  & t \in C_x
		\end{cases}
		\]
		is a Borel parametrization of $X \times [0,1]$. Denote $\pr_2$ the projection on the second component. It is easy to see that ${\pr_2}_\ast\widetilde{G}(x,\cdot)_\ast\pi^x=\lambda$ for all $x \in X$. Hence, the mapping $G := \pr_2 \circ \widetilde{G}$ has the desired properties.
	\end{proof}
	We state two corollaries of Theorem~\ref{thm:borel_iso_parametr} that are useful throughout the paper:
	\begin{cor}\label{cor:borel_iso_parametr2}
		Let $X, Y$ and $Z$ be standard Borel spaces, $\mu$ a kernel from $Z$ to $X$ and $\nu$ a kernel from $Z$ to $Y$ s.t.\ $\mu^z$ and $\nu^z$ are continuous probability measures for all $z \in Z$. Then there exists a Borel function 
		\[
		G : Z \times X \to Y
		\] 
		such that for all $z \in Z$ the mapping $G^z = G(z, \cdot ) : X \to Y$ is a Borel ismorphism satisfying $G^z_\ast\mu^z = \nu^z$.
	\end{cor}

	\begin{proof}
		As we see at the end of the proof of Theorem~\ref{thm:borel_iso_parametr}, there are Borel parametrizations  $\widetilde F : Z \times X \to Z \times [0,1]$ and $\widetilde H: Z \times Y \to Z \times [0,1]$ such that ${\pr_2}_\ast\widetilde{F}(x,\cdot)_\ast\mu^x=\lambda$ and ${\pr_2}_\ast\widetilde{H}(x,\cdot)_\ast\nu^x=\lambda$ 
		for all $z \in Z$. It is easy to see that $G:= \pr_2 \circ \widetilde H^{-1} \circ \widetilde F$ has the desired properties.
	\end{proof}

	\begin{cor}\label{cor:borel_iso_parametr}
		Let $X$ and $Y$ be standard Borel spaces and $\pi$ a kernel from $X$ to $Y$. Then there exists a Borel function
		\[
		G : X \times Y \times [0,1]  \to [0,1]^2
		\] 
		such that for all $x \in X$ the mapping $G^x= G(x, \cdot ) : Y \times [0,1] \to [0,1]^2$ is a Borel isomorphism satisfying $G^x_\ast(\pi^x \otimes \lambda ) = \lambda^2$, where $\lambda^2$ denotes the Lebesgue measure on $[0,1]^2$.
	\end{cor}
	\begin{proof}
		Note that $\pi^x \otimes \lambda$ is a continuous probability measure on $Y \times [0,1]$ for any  $\pi^x \in \prob(Y)$. Hence we can apply Theorem~\ref{thm:borel_iso_parametr} to standard Borel spaces $X$ and $Y \times [0,1]$ and the kernel $(\pi^x\otimes \lambda)_{x \in X}$.
	\end{proof}
	A further useful application of the isomorphism theorem for kernels is the following proposition:
	\begin{prop}\label{prop:part_pushfwd2}
		Let $X$ and $Z$ be Polish spaces and for every $n \in \N$, let $M^n \subset Z \times X$ be Borel such that for every $z \in Z$ the collection $\M_z:= \{ M^n_z : n \in \N \}$ is a partition of $X$. Further let $\mu$ be a kernel from $Z$ to $X$ s.t.\ $\mu^z$ is continuous for all $z \in Z$. 	
	Then there exists a Borel mapping $\Phi : Z \times X \to X \times [0,1]$ s.t.\ for all $z \in Z$ the mapping $\Phi^z =\Phi(z,\cdot) : X \to X \times [0,1]$ is a Borel isomorphism satisfying $\Phi^z _\ast(\mu^z|_{M^n_z})=(\mu^z|_{M^n_z}) \otimes \lambda$ for all $n \in \N$. 
	\end{prop}
	
	\begin{proof}
        For $n \in \N$, write $Z^n := \{z \in Z : \mu^z(M^n_z) >0 \}$ and $\tilde M^n := M^n \cap (Z^n \times X)$. Note that $Z^n$ is Borel by \cite[Theorem~17.25]{Ke95}. We first show that there is a Borel map $\Phi_n : \tilde M^n \to X \times [0,1]$ such that for every $z \in Z^n$, the map $\Phi_n(z,\cdot) :  M^n_z \to M^n_z \times [0,1]$ is bijective and satisfies ${\Phi_n^z}_\ast (\mu^z|_{M^n_z}) = \mu^z|_{M^n_z} \otimes \lambda$. 

        Indeed, by Theorem~\ref{thm:Borel_param}(iii) there exists a Borel isomorphism $f_n : {\tilde M^n} \to Z_n \times [0,1]$ such that $f_n( \{z\} \times M^n_z ) = \{z\} \times [0,1]$ for all $z \in Z^n$. Moreover, by Corollary~\ref{cor:borel_iso_parametr2} there is a map $G_n : Z^n \times [0,1] \to [0,1]^2$ such that $G^z$ is a Borel isomorphism satisfying ${G_n^z}_\ast {( \pr_{[0,1]} \circ f^z_n)}_\ast (\mu^z|_{M^n_z}) =  {( \pr_{[0,1]} \circ f^z_n)}_\ast (\mu^z|_{M^n_z}) \otimes \lambda$. We set $\Phi^z_n := (( \pr_{[0,1]} \circ f^z_n)^{-1},\id) \circ G^z_n \circ ( \pr_{[0,1]} \circ f^z_n)$.
    
        We write $E:= \bigcup_{n\in \N}  { \tilde M^n} $ and $\tilde \Phi := \bigcup_{n \in \N} \Phi_n$. Note that $\tilde \Phi : E \to X \times [0,1]$ and for every $z \in Z$, $\tilde\Phi^z : E_z \to E_z \times [0,1]$ is  a Borel isomorphism satisfying $\tilde \Phi^z _\ast(\mu^z|_{M^n_z})=(\mu^z|_{M^n_z}) \otimes \lambda$ for all $n \in \N$.
        
        Now, it remains to modify $\tilde \Phi$ on null sets to define a map $\Phi$ that is defined on the entire space $Z \times X$ and such that $\Phi^z : X \to X \times [0,1]$ is a Borel isomorphism with the desired properties. By Proposition~\ref{prop:perfect_nullset} there are Borel sets $A \subset Z \times X$ and $B \subset Z \times X \times [0,1]$ such that for all $z \in Z$ the sets $A_z$ and $B_z$ are compact perfect and satisfy $\mu^z(A_z)=0$ and $\mu^z\otimes \lambda(B_z)=0$. 

        We write $\bar{\Phi} := (\pr_Z,\tilde\Phi) : E \to Z \times X \times [0,1]$. We consider the sets
        \[
        C:= A \cup \bar{\Phi}^{-1}(B) \cup (Z \times X\setminus E), \qquad D:= \bar{\Phi}(A) \cup B \cup (Z \times X \times [0,1] \setminus \bar{\Phi}(E)   ) .
        \]
		Clearly, $C$ and $D$ both satisfy the assumptions from Theorem~\ref{thm:Borel_param}, so there exists a Borel isomorphism $\Psi: C \to D$ such that $\Psi(\{z\} \times C_z ) = \{z\} \times D_z$  for all $z \in Z$. Denote $\pr : Z \times X \times [0,1] \to X \times [0,1]$ the projection. 
		It is easy to check that the mapping
		\[
		\Phi: Z \times X \to X \times [0,1] : (z,x) \mapsto \begin{cases}
			\pr(\bar\Phi(z,x)) & (z,x) \notin C, \\
			\pr(\Psi(z,x)) & (z,x) \in C,
		\end{cases} 
		\] 
		has the desired properties.
	\end{proof}

\bibliographystyle{abbrv} 
\bibliography{joint_biblio}

\begin{thebibliography}{10}

\bibitem{AcBaZa20}
B.~Acciaio, J.~{Backhoff-Veraguas}, and A.~Zalashko.
\newblock Causal optimal transport and its links to enlargement of filtrations and continuous-time stochastic optimization.
\newblock {\em Stoch.~Proc.~Appl.}, 130(5):2918--2953, 2020.

\bibitem{AcBePa20}
B.~Acciaio, M.~Beiglb\"{o}ck, and G.~Pammer.
\newblock Weak transport for non-convex costs and model-independence in a fixed-income market.
\newblock {\em Math. Finance}, 31(4):1423--1453, 2021.

\bibitem{Al81}
D.~J. Aldous.
\newblock Weak convergence and general theory of processes.
\newblock Unpublished monograph: Department of Statistics, University of California, Berkeley, 1981.

\bibitem{AlBo07}
C.~D. Aliprantis and K.~C. Border.
\newblock {\em Infinite Dimensional Analysis - A Hitchhiker's Guide}.
\newblock Springer Science \& Business Media, Berlin Heidelberg, 2007.

\bibitem{Am03}
L.~Ambrosio.
\newblock Lecture notes on optimal transport problems.
\newblock In {\em Mathematical aspects of evolving interfaces ({F}unchal, 2000)}, volume 1812 of {\em Lecture Notes in Math.}, pages 1--52. Springer, Berlin, 2003.

\bibitem{BaBaBeEd19a}
J.~{Backhoff-Veraguas}, D.~Bartl, M.~Beiglb\"{o}ck, and M.~Eder.
\newblock Adapted {W}asserstein distances and stability in mathematical finance.
\newblock {\em Finance Stoch.}, 24(3):601--632, 2020.

\bibitem{BaBaBeEd19b}
J.~{Backhoff-Veraguas}, D.~Bartl, M.~Beiglb\"{o}ck, and M.~Eder.
\newblock All adapted topologies are equal.
\newblock {\em Probab.~Theory Relat.~Fields}, 178(3-4):1125--1172, 2020.

\bibitem{BaBeLiZa17}
J.~{Backhoff-Veraguas}, M.~Beiglb\"{o}ck, Y.~Lin, and A.~Zalashko.
\newblock Causal transport in discrete time and applications.
\newblock {\em SIAM J.~Optim.}, 27(4):2528--2562, 2017.

\bibitem{BaBePa21}
D.~Bartl, M.~Beiglb{\"o}ck, and G.~Pammer.
\newblock The {W}asserstein space of stochastic processes.
\newblock {\em J.~Eur.~Math.~Soc.}, 2024.
\newblock To appear. arXiv:2104.14245.

\bibitem{BaDoGu20}
E.~Bayraktar, Y.~Dolinsky, and J.~Guo.
\newblock Continuity of utility maximization under weak convergence.
\newblock {\em Math. Financ. Econ.}, 14(4):725--757, 2020.

\bibitem{BaDoDo20}
E.~Bayraktar, L.~Dolinskyi, and Y.~Dolinsky.
\newblock Extended weak convergence and utility maximisation with proportional transaction costs.
\newblock {\em Finance Stoch.}, 24(4):1013--1034, 2020.

\bibitem{BeLa20}
M.~Beiglb{\"o}ck and D.~Lacker.
\newblock Denseness of adapted processes among causal couplings.
\newblock {\em arXiv:1805.03185}, 2020.

\bibitem{BiTa19}
J.~Bion-Nadal and D.~Talay.
\newblock On a {W}asserstein-type distance between solutions to stochastic differential equations.
\newblock {\em Ann.~Appl.~Probab.}, 29(3):1609--1639, 2019.

\bibitem{BoLiOb23}
P.~Bonnier, C.~Liu, and H.~Oberhauser.
\newblock Adapted topologies and higher rank signatures.
\newblock {\em Ann.~Appl.~Probab.}, 33(3):2136--2175, 2023.

\bibitem{Do14}
Y.~Dolinsky.
\newblock Hedging of game options under model uncertainty in discrete time.
\newblock {\em Electron. Commun. Probab.}, 19:no. 19, 11, 2014.

\bibitem{FiGl21}
A.~Figalli and F.~Glaudo.
\newblock {\em An invitation to optimal transport, {W}asserstein distances, and gradient flows}.
\newblock EMS Textbooks in Mathematics. EMS Press, Berlin, 2021.

\bibitem{Ga99}
W.~Gangbo.
\newblock The {M}onge mass transfer problem and its applications.
\newblock {\em Contemporary Mathematics}, 226:79--104, 1999.

\bibitem{Gi04}
N.~Gigli.
\newblock {\em On the geometry of the space of probability measures in {$\mathbb R^n$} endowed with the quadratic optimal transport distance}.
\newblock PhD thesis, Scuola Normale Superiore di Pisa, 2004.

\bibitem{GlPfPi17}
M.~Glanzer, G.~C. Pflug, and A.~Pichler.
\newblock Incorporating statistical model error into the calculation of acceptability prices of contingent claims.
\newblock {\em Math. Program.}, 174(1-2, Ser. B):499--524, 2019.

\bibitem{He96}
M.~F. Hellwig.
\newblock Sequential decisions under uncertainty and the maximum theorem.
\newblock {\em J. Math. Econom.}, 25(4):443--464, 1996.

\bibitem{Ho91}
D.~Hoover.
\newblock Convergence in distribution and {S}korokhod convergence for the general theory of processes.
\newblock {\em Probab.~Theory Relat.~Fields}, 89(3):239--259, 1991.

\bibitem{HoKe84}
D.~N. Hoover and H.~J. Keisler.
\newblock Adapted probability distributions.
\newblock {\em Transactions of the American Mathematical Society}, 286(1):159--201, 1984.

\bibitem{Kall02}
O.~Kallenberg.
\newblock {\em Foundations of Modern Probability}.
\newblock Probability and its Applications. Springer-Verlag, New York, second edition, 2002.

\bibitem{Ke95}
A.~S. Kechris.
\newblock {\em Classical descriptive set theory}, volume 156 of {\em Graduate Texts in Mathematics}.
\newblock Springer-Verlag, New York, 1995.

\bibitem{KiPfPi20}
K.~B. Kirui, G.~C. Pflug, and A.~Pichler.
\newblock New algorithms and fast implementations to approximate stochastic processes.
\newblock {\em arXiv:2012.01185}, 2020.

\bibitem{Ku07}
T.~Kurtz.
\newblock The {Y}amada--{W}atanabe--{E}ngelbert theorem for general stochastic equations and inequalities.
\newblock {\em Electron. J. Probab}, 12:951--965, 2007.

\bibitem{La18}
R.~{Lassalle}.
\newblock {Causal transference plans and their Monge--Kantorovich problems}.
\newblock {\em Stoch.~Anal.~Appl.}, 36(3):452--484, 2018.

\bibitem{Ma79}
R.~D. Mauldin.
\newblock Borel parametrizations.
\newblock {\em Transactions of the American Mathematical Society}, 250:223--234, 1979.

\bibitem{NiSu20}
F.~Nielsen and K.~Sun.
\newblock {\em Chain Rule Optimal Transport}, pages 191--217.
\newblock Springer International Publishing, Cham, 2021.

\bibitem{PfPi12}
G.~C. Pflug and A.~Pichler.
\newblock A distance for multistage stochastic optimization models.
\newblock {\em SIAM J.~Optim.}, 22(1):1--23, 2012.

\bibitem{PfPi14}
G.~C. Pflug and A.~Pichler.
\newblock {\em Multistage Stochastic Optimization}.
\newblock Springer Series in Operations Research and Financial Engineering. Springer, Cham, 2014.

\bibitem{PfPi15}
G.~C. Pflug and A.~Pichler.
\newblock Dynamic generation of scenario trees.
\newblock {\em Comput. Optim. Appl.}, 62(3):641--668, 2015.

\bibitem{PiSh21}
A.~Pichler and A.~Shapiro.
\newblock Mathematical foundations of distributionally robust multistage optimization.
\newblock {\em SIAM J. Optim.}, 31(4):3044--3067, 2021.

\bibitem{Pr07b}
A.~Pratelli.
\newblock On the equality between {M}onge's infimum and {K}antorovich's minimum in optimal mass transportation.
\newblock {\em Ann. Inst. H. Poincar\'{e} Probab. Statist.}, 43(1):1--13, 2007.

\bibitem{Ru85}
L.~R\"{u}schendorf.
\newblock The {W}asserstein distance and approximation theorems.
\newblock {\em Z. Wahrsch. Verw. Gebiete}, 70(1):117--129, 1985.

\bibitem{Sa15}
F.~Santambrogio.
\newblock {\em Optimal Transport for Applied Mathematicians}, volume~87 of {\em Progress in Nonlinear Differential Equations and their Applications}.
\newblock Birkh\"auser Cham, 2015.
\newblock Calculus of variations, PDEs, and modeling.

\bibitem{Ve70}
A.~M. Vershik.
\newblock Decreasing sequences of measurable partitions and their applications.
\newblock {\em Sov. Mat. Dokl.}, 11(4):1007 -- 1011, 1970.

\bibitem{Ve94}
A.~M. Vershik.
\newblock Theory of decreasing sequences of measurable partitions.
\newblock {\em Algebra i Analiz}, 6(4):1--68, 1994.

\bibitem{Vi03}
C.~Villani.
\newblock {\em Topics in Optimal Transportation}, volume~58 of {\em Graduate Studies in Mathematics}.
\newblock American Mathematical Society, Providence, RI, 2003.

\bibitem{Vi09}
C.~Villani.
\newblock {\em Optimal Transport, Old and New}, volume 338 of {\em Grundlehren der mathematischen Wissenschaften}.
\newblock Springer, 2009.

\bibitem{YW}
T.~Yamada and S.~Watanabe.
\newblock On the uniqueness of solutions of stochastic differential equations.
\newblock {\em J.~Math.~Kyoto Univ.}, 11(1):155--167, 1971.

\end{thebibliography}
\end{document}